\documentclass{amsart}
\usepackage{amssymb,amscd}
\advance\hoffset by -1in
\advance\voffset by -1in

\usepackage[cp1251]{inputenc}
\usepackage[russian]{babel}
\newtheorem{theorem}{Теорема}
\newtheorem{lemma}{Лемма}

\newtheorem{proposition}{Предложение}

\title{Теоремы о свободе для свободных лиевых сумм алгебр Ли}

\author{А.\,Ф.\,Красников}
\address{Омский государственный университет, г.\,Омск}
\email{phomsk@mail.ru}

\begin{document}

\maketitle

\section*{Введение}
\noindent Известная теорема о свободе Ширшова \cite{Sh} показывает, что если $F$ --- свободная алгебра Ли с множеством $R$ образующих
и одним определяющим соотношением $s=0$, в левую часть которого входит образующий $x$, то подалгебра, порожденная в алгебре $F$ множеством $R\setminus x$, свободна.\\
Подобное утверждение имеет место для $k\text{-ступенно}$ разрешимых и полинильпотентных алгебр Ли \cite{Tl_0}, \cite{Tl}.\\
Пусть $F$ --- свободная
сумма нетривиальных алгебр Ли $A_i~(i\in I)$. Эндоморфизм $\varphi$ алгебры $F$ будем называть элементарным, если найдется $K_\varphi\subseteq I$
такое, что если $a\in A_i$, то $\varphi(a)=0~(i\in K_\varphi)$, $\varphi(a)=a~(i\not\in K_\varphi)$.\\
В настоящей работе доказана теорема:
\begin{theorem}\label{tm2_algr}
Пусть $F$ --- свободная
сумма нетривиальных алгебр Ли $A_i$,
$N$ --- допускающий элементарные эндоморфизмы алгебры $F$ идеал алгебры $F$ такой, что $N\cap A_i=0$ $(i=1,\ldots,n;~n>2)$,
\begin{eqnarray}\label{end_algr_2}
N=N_{11} \geqslant \ldots \geqslant N_{1,m_1+1}=N_{21} \geqslant \ldots \geqslant N_{s,m_s+1},
\end{eqnarray}
где $N_{kl}$ --- l-я степень алгебры $N_{k1}$.\\
Пусть, далее, $r\in N_{1i}\backslash N_{1,i+1}\,(i\leqslant m_1)$, $R$ --- идеал, порожденный в алгебре $F$ элементом $r$, $H$ --- свободная
сумма алгебр $A_i~(i=1,\ldots,n-1)$.\\
Если (и только если) $r\not\in H+N_{1,i+1}$, то $H\cap (R+N_{kl})=H\cap N_{kl}$, где $N_{kl}$ --- произвольный член ряда {\rm (\ref{end_algr_2})}.
\end{theorem}
\noindent Теорема Харлампович \cite{Hm} показывает, что если алгебра Ли $G$ задана с помощью образующих элементов $x_1,\ldots,x_n$ и определяющих соотношений $r_1,\ldots,r_m$, где $n>m$, то среди
данных образующих существуют такие $n-m$ элементов $x_{i_1},\ldots,x_{i_{n-m}}$, которые свободно порождают свободную подалгебру в $G$.\\
Подобное утверждение имеет место для разрешимых алгебр Ли, т.е. для алгебр Ли, заданных с помощью образующих
элементов и определяющих соотношений в многообразиях разрешимых алгебр Ли ступени разрешимости не выше $k$ \cite{Hm}.\\
В настоящей работе доказана теорема:
\begin{theorem}\label{tm3_algr}
Пусть $F$ --- свободная
сумма нетривиальных алгебр Ли $A_i$,
$N$ --- допускающий элементарные эндоморфизмы алгебры $F$ идеал алгебры $F$ такой, что $N\cap A_i=0$ $(i=1,\ldots,n)$,
$F/N$ --- разрешимая алгебра,
\begin{eqnarray}\label{end_algr_3}
N=N_{11} \geqslant \ldots \geqslant N_{1,m_1+1}=N_{21} \geqslant \ldots \geqslant N_{s,m_s+1},
\end{eqnarray}
где $N_{kl}$ --- l-я степень алгебры $N_{k1}$.\\
Пусть, далее, $r_1,\ldots,r_m$ --- элементы из $N$ $(m<n)$, $R$ --- идеал, порожденный в алгебре $F$ элементами $r_1,\ldots,r_m$.\\
Тогда найдется $J \subseteq \{1,2,\ldots,n\}$, $|J|\geqslant n-m$,
такое, что для подалгебры $H$ алгебры $F$, порожденной алгебрами $A_i$ $(i\in J)$ и произвольного $N_{kl}$ из ряда {\rm (\ref{end_algr_3})} справедлива формула
$H\cap (R+N_{kl}) = H\cap N_{kl}$.
\end{theorem}

\section{Некоторые свойства производных Фокса}
\noindent Все алгебры будут рассматриваться над произвольным фиксированным полем $P$. Пусть $L$ --- алгебра Ли. Через $U(L)$
будем обозначать универсальную обертывающую алгебру алгебры $L$, через $L_{(k)}$ --- $k$-ю степень $L$.\\
Пусть $u$, $v\in L$. Алгебра $L$ вкладывается в $U(L)$ и $[u,v]=uv-vu$ в $U(L)$ (мы обозначаем через $[u,v]$
умножение в $L$). Если $M$ --- идеал в
$L$, то через $M_U$ будем обозначать идеал, порожденный $M$ в $U(L)$. Если $M$ --- идеал алгебры $L$, порожденный
множеством $X$, то будем писать $M=\mbox{ид}_L (X)$.\\
Пусть $A$, $B$ --- подмножества множества элементов алгебры $U(L)$. Через $AB$ будем обозначать множество сумм произведений вида $ab$, где $a$, $b$ пробегают соответственно элементы $A$, $B$.\\
Теорема Пуанкаре-Биркгофа-Витта показывает, что
если $u_1,\ldots,u_n,\ldots$ --- упорядоченный базис в $L$, то $1$ и одночлены вида $u_{i_1}\cdots u_{i_r}$, где $i_1\leqslant\ldots\leqslant i_r$, образуют базис в $U(L)$, который называется стандартным базисом в $U(L)$.\\
Пусть $F=(\underset{i\in I} \sum^{\ast} A_i)\ast G$ --- свободная
сумма нетривиальных алгебр Ли $A_i~(i\in I)$ и свободной алгебры Ли $G$ с множеством свободных порождающих
$\{g_j | j\in J\}$; $N$ --- идеал в $F$.
Отметим, что если $H$ ---  подалгебра алгебры $F$ такая, что $H\cap A_i=0~(i\in I)$,
то $H$ --- свободная алгебра Ли \cite{Sh2}.\\
Заметим сразу, что множество свободных порождающих алгебры
мы иногда будем еще называть ее базой.\\
Если $u\in U(F)$, то однозначно находятся элементы $D_k(u)\in U(F)$
такие, что
\begin{eqnarray*}
u = \underset{i\in I} \sum D_i(u) + \underset{j\in J} \sum g_j D_j(u),
\end{eqnarray*}
где $D_i(u)\in A_i U(F)$.
Следуя Умирбаеву \cite{Um}, назовем $D_k(u)$ $(k\in I\cup J)$ производными Фокса элемента $u$. Нетрудно видеть, что производные Фокса обладают
следующими свойствами:
\begin{gather}
D_k(\alpha u+\beta v)=\alpha D_k(u)+\beta D_k(v)~(k\in I\cup J);\notag\\
D_j(g_j)=1~(j\in J),~D_k(g_j)=0,~\mbox{если}~k\neq j;\notag\\
\mbox{если}~a_i\in A_i, ~\mbox{то}~ D_i(a_i)=a_i~(i\in I),~D_k(a_i)=0,~\mbox{если}~k\neq i;\notag\\
D_k([u,v])= D_k(u)v-D_k(v)u;~D_k([n,u])\equiv D_k(n)u\mod{N_U}~(k\in I\cup J);\notag
\end{gather}
где $u,~v\in F,~n\in N,~\alpha,~\beta\in P$.\\
Пусть $H$ ---  подалгебра алгебры $F$ такая, что $H\cap A_i=0~(i\in I)$, $f\in H$,
$\{h_k | k\in K\}$ --- база $H$, $\{\partial_k | k\in K\}$ --- соответствующие этой базе производные Фокса алгебры $H$.
Из $f=\sum_k h_k\partial_k(f)$ и $h_k=\sum_{i\in I} D_i(h_k)+\sum_{j\in J} g_jD_j(h_k)$ следует, что\\
$f=\sum_{i\in I} (\sum_k D_i(h_k)\partial_k(f))+\sum_{j\in J} g_j(\sum_k D_j(h_k)\partial_k(f))$.\\
Поэтому $D_l(f)=\sum_k D_l(h_k)\partial_k(f)$.

\begin{lemma}\label{alg1_lm1}
Пусть $F=(\underset{i\in I} \sum^{\ast} A_i)\ast G$ --- свободная
сумма нетривиальных алгебр Ли $A_i~(i\in I)$ и свободной алгебры Ли $G$ с базой
$\{g_j | j\in J\}$; $K \subseteq I\cup J$, $H$ --- свободная
сумма свободной алгебры Ли с базой $\{g_j | j\in K\}$
и алгебр $A_i~(i\in K)$; $u_k$ --- элементы из $U(H)~(k\in J\cap K)$,
$u_m$ --- элементы из $A_mU(H)~(m\in I\cap K)$, почти все равные нулю; $N$ --- идеал в $F$.
Если
\begin{eqnarray}\label{alg1_lm1_1}
\underset{m\in I\cap K} \sum u_m + \underset{k\in  J\cap K} \sum g_ku_k\equiv 0\mod{N_U},
\end{eqnarray}
то найдется $v\in H\cap N$ такой, что $D_k(v)\equiv u_k \mod{N_U}$ $(k\in K)$.
\end{lemma}
\begin{proof}
Пусть $\{a_i,~i\in J_1\}$ --- базис пространства $H\cap N$.
Дополним его элементами $\{b_j,~j\in J_2\}$ до базиса
$H$, потом базис $H$ дополним до базиса $H+N$ элементами $\{c_s,~s\in
J_3\}$ из $N$.  Базис $H+N$ дополним до базиса $F$ элементами $\{d_t,~t\in
J_4\}$. Положив $a_i<b_j<c_s<d_t$
получаем стандартный базис алгебры $U(F)$, состоящий из слов
\begin{eqnarray}\label{alg1_1}
a_{i_1}\ldots a_{i_\mu}b_{j_1}\ldots b_{j_\nu}c_{s_1}\ldots
c_{s_\eta}d_{t_1}\ldots d_{t_\theta},
\end{eqnarray}
где $i_1\leqslant\ldots\leqslant i_\mu,~j_1\leqslant\ldots\leqslant
j_\nu,~s_1\leqslant\ldots\leqslant s_\eta,~t_1\leqslant\ldots\leqslant
t_\theta,~\mu\geqslant 0,~\nu\geqslant 0,~\eta\geqslant 0,~\theta\geqslant 0$.
Обозначим $\underset{m\in I\cap K} \sum u_m + \underset{k\in  J\cap K} \sum g_ku_k$ через $u^\prime$.
 Из (\ref{alg1_lm1_1}) следует, что $u^\prime$ ---
линейная комбинация одночленов вида (\ref{alg1_1}),
для которых $\mu\geqslant 1$, $\eta=\theta=0$.
Следовательно, $u^\prime = \sum_{x\in X} n_xw_{x1}\ldots w_{xz_x}$, где $n_x\in H\cap N,~w_{pq}\in H,~z_x\geqslant 0$.
Полагаем $v = \sum_{x\in X} [\ldots[n_xw_{x1}]\ldots
w_{xz_x}]$. Тогда $v\in H\cap N$ и
\begin{eqnarray}\label{alg1_lm1_2}
D_m(v) \equiv \sum_{x\in X} D_m(n_x)w_{x1}\ldots w_{xz_x}\mod{N_U}~(m\in K).
\end{eqnarray}
Будем иметь
\begin{gather}
\underset{m\in I\cap P} \sum\sum_{x\in X} D_m(n_x)w_{x1}\ldots w_{xz_x} + \underset{k\in  J\cap P} \sum g_k\sum_{x\in X} D_k(n_x)w_{x1}\ldots w_{xz_x}=\notag\\
\sum_{x\in X}(\underset{m\in I\cap P} \sum D_m(n_x)+ \underset{k\in  J\cap P} \sum g_k D_k(n_x))w_{x1}\ldots w_{xz_x} =\notag\\
\sum_{x\in X}n_xw_{x1}\ldots w_{xz_x}=\underset{m\in I\cap P} \sum u_m + \underset{k\in  J\cap P} \sum g_ku_k.\notag
\end{gather}
Таким образом,
\begin{eqnarray}\label{alg1_lm1_3}
\sum_{x\in X} D_k(n_x)w_{x1}\ldots w_{xz_x}=u_k~(k\in K).
\end{eqnarray}
Из (\ref{alg1_lm1_2}), (\ref{alg1_lm1_3}) следует $D_k(v) \equiv u_k\mod{N_U}~(k\in K)$.
\end{proof}

\begin{lemma}\label{alg1_tm1}
Пусть $F=(\underset{i\in I} \sum^{\ast} A_i)\ast G$ --- свободная
сумма нетривиальных алгебр Ли $A_i~(i\in I)$ и свободной алгебры Ли $G$ с базой
$\{g_j | j\in J\}$; $K \subseteq I\cup J$, $H$ --- свободная
сумма свободной алгебры Ли с базой $\{g_j | j\in K\}$
и алгебр $A_i~(i\in K)$; $u_k$ --- элементы из $U(F)~(k\in J\cap K)$,
$u_m$ --- элементы из $A_mU(F)~(m\in I\cap K)$, почти все равные нулю; $N$ --- идеал в $F$.
Если
\begin{eqnarray}\label{alg1_tm1_1}
\underset{m\in I\cap K} \sum u_m + \underset{k\in  J\cap K} \sum g_ku_k\equiv 0\mod{N_U},
\end{eqnarray}
то найдется $v\in \mbox{\rm ид}_F (H\cap N)$ такой, что $D_k(v)\equiv u_k \mod{N_U}$  $(k\in K)$.
\end{lemma}
\begin{proof}
Выберем в алгебре $U(F)$ стандартный базис, состоящий из слов (\ref{alg1_1}). Найдутся попарно различные стандартные одночлены
$f_1,\ldots,f_z$, для которых $\mu=\nu=\eta=0$, такие, что
\begin{gather}
u_k \equiv \sum_{l=1}^z u_{kl}f_l\mod{N_U}~(k\in J\cap K),\label{alg1_tm1_2}\\
u_m \equiv \sum_{l=1}^z u_{ml}f_l\mod{N_U}~(m\in I\cap K),\label{alg1_tm1_20}
\end{gather}
где $u_{kl}\in U(H)$; $u_{ml}\in A_mU(H)$.
Формулы (\ref{alg1_tm1_1}), (\ref{alg1_tm1_2}), (\ref{alg1_tm1_20}) показывают, что
\begin{eqnarray*}
\sum_{l=1}^z (\underset{m\in I\cap K} \sum  u_{ml} + \underset{k\in  J\cap K} \sum g_k u_{kl})f_l\equiv 0\mod{N_U},
\end{eqnarray*}
следовательно
\begin{eqnarray}\label{alg1_tm1_3}
\underset{m\in I\cap K} \sum  u_{ml} + \underset{k\in  J\cap K} \sum g_k u_{kl} \equiv 0 \mod{N_U}~(1\leqslant l\leqslant z).
\end{eqnarray}
Ввиду леммы \ref{alg1_lm1} из (\ref{alg1_tm1_3}) следует существование
элементов $v_1,\ldots, v_z$ из $H\cap N$ таких, что
\begin{eqnarray}\label{alg1_tm1_4}
D_k(v_l)\equiv u_{kl} \mod{N_U}~(k\in K).
\end{eqnarray}
Пусть $f_l=d_{l1}\ldots d_{lz_l}$. Полагаем $v =\sum_{l=1}^z [\ldots[v_ld_{l1}]\ldots d_{lz_l}]$.\\
Тогда
$v\in\mbox{ид}_F (H\cap N)$ и
\begin{eqnarray}\label{alg1_tm1_5}
D_k(v) \equiv \sum_{l=1}^z D_k(v_l)f_l\mod{N_U}~(k\in K).
\end{eqnarray}
Из (\ref{alg1_tm1_4}), (\ref{alg1_tm1_5}) следует
\begin{eqnarray}\label{alg1_tm1_6}
D_k(v) \equiv \sum_{l=1}^z u_{kl}f_l\mod{N_U}~(k\in K).
\end{eqnarray}
Из (\ref{alg1_tm1_2}), (\ref{alg1_tm1_20}), (\ref{alg1_tm1_6}) видно, что $D_k(v)\equiv u_k\mod{N_U}$ $(k\in K)$.
\end{proof}

\begin{lemma}\label{alg1_tm1_lm3}
Пусть $F=(\underset{i\in I} \sum^{\ast} A_i)\ast G$ --- свободная
сумма нетривиальных алгебр Ли $A_i~(i\in I)$ и свободной алгебры Ли $G$ с базой
$\{g_j | j\in J\}$; $N$ --- идеал в $F$, $M = \mbox{\rm ид}_F(N\cap A_i|i\in I)+[N,N]$, $v\in F$.\\
Если
$D_k(v)\equiv ~0\mod{N_U},~k\in I\cup J$, то $v\in M$.
\end{lemma}
\begin{proof}
Так как $v\in N_U$, то образ $v$ в $U(F/N)$=$U(F)/N_U$ равен $0$, откуда $v\in N$.
Пусть $u_1,\ldots,u_n,\ldots$ --- базис в $N/M$. Дополним его до базиса $F/M$
элементами  $v_1,\ldots,v_n,\ldots$. Положив $v_i<u_j$
получаем стандартный базис алгебры $U(F/M)$, состоящий из слов
\begin{eqnarray}\label{alg1_tm1_lm3_2}
v_{i_1}\ldots v_{i_\mu}u_{j_1}\ldots u_{j_\nu},
\end{eqnarray}
$i_1\leqslant\ldots\leqslant i_\mu,~j_1\leqslant\ldots\leqslant j_\nu$.
Предположим, $v\not\in M$. Тогда
\begin{eqnarray}\label{alg1_tm1_lm3_1}
v+M =\alpha_1 u_{j_1}+\ldots +\alpha_m u_{j_m},
\end{eqnarray}
$0\neq \alpha_k\in P$, $k=1,\ldots ,m$. Образы $D_k(v)~(k\in I\cup J)$ при естественном гомоморфизме
$U(F)\rightarrow U(F/M)$ обозначим через $w_k$.\\
Покажем, что $(g_k+M) w_k~(k\in J)$ --- линейные комбинации стандартных одночленов (\ref{alg1_tm1_lm3_2}) длины $\geqslant 2$.
Действительно, если $g_k\in N$, то в разложении $(g_k+M) w_k$ по базису появятся только стандартные базисные одночлены
(\ref{alg1_tm1_lm3_2}) у которых $\nu\geqslant 2$, а если $g_k\not\in N$,
то в разложении $(g_k+M) w_k$ по базису будут участвовать только стандартные базисные одночлены
(\ref{alg1_tm1_lm3_2}) у которых либо $\nu\geqslant 2$ либо $\mu+\nu\geqslant 2$.\\
Покажем, что $w_k~(k\in I)$ --- линейные комбинации стандартных одночленов (\ref{alg1_tm1_lm3_2}) длины $\geqslant 2$.
Дополним базис $u_1,\ldots,u_n,\ldots$ до базиса $(A_k+N)/M$
элементами  $a_1,\ldots,a_n,\ldots$ из $(A_k+M)/M$. Дополним базис $(A_k+N)/M$ до базиса $F/M$
элементами  $b_1,\ldots,b_n,\ldots$. Положив $a_i<b_s<u_j$
получаем стандартный базис алгебры $U(F/M)$, состоящий из слов
\begin{eqnarray}\label{alg1_tm1_lm3_3}
a_{i_1}\ldots a_{i_\delta}b_{s_1}\ldots b_{s_\eta}u_{j_1}\ldots u_{j_\gamma},
\end{eqnarray}
$i_1\leqslant\ldots\leqslant i_\delta,~s_1\leqslant\ldots\leqslant s_\eta,~j_1\leqslant\ldots\leqslant j_\gamma$.
Пусть $w\in \{w_k | k\in I\}$, $w = \alpha_1c_1d_1+\ldots+\alpha_lc_ld_l$, $\alpha_i\in P$, $c_i\in \{a_1,\ldots,a_n,\ldots\}$, $c_id_i$ --- попарно различные одночлены вида (\ref{alg1_tm1_lm3_3}).\\
Так как $D_k(v)\equiv ~0\mod{N_U}$ и $N\cap A_k\subseteq M$, то
$d_i$ --- одночлены вида (\ref{alg1_tm1_lm3_3}) у которых $\gamma > 0$, поэтому $c_id_i$ --- линейная комбинация стандартных одночленов (\ref{alg1_tm1_lm3_2}) длины $\geqslant 2$.\\
Мы показали, что в разложении $v+M$ по базису, состоящему из слов (\ref{alg1_tm1_lm3_2}), будут участвовать только стандартные базисные одночлены длины $\geqslant 2$, что противоречит равенству (\ref{alg1_tm1_lm3_1}).
Поэтому $v\in M$.
\end{proof}

\begin{theorem}\label{alg1_tm2}
Пусть $F=(\underset{i\in I} \sum^{\ast} A_i)\ast G$ --- свободная
сумма нетривиальных алгебр Ли $A_i~(i\in I)$ и свободной алгебры Ли $G$ с базой
$\{g_j | j\in J\}$; $N$ --- идеал в $F$, $M = \mbox{\rm ид}_F(N\cap A_i|i\in I)+[N,N]$;
$K \subseteq I\cup J$, $H$ --- свободная
сумма свободной алгебры Ли с базой $\{g_j | j\in K\}$
и алгебр $A_i~(i\in K)$, $v\in F$.
Тогда
\begin{eqnarray}\label{alg1_tm2_1}
D_k(v)\equiv ~0~ mod ~N_U,~k\not\in K,
\end{eqnarray}
если и только если $v\equiv v_0 + v_1\mod{M}$, $v_0\in H$, $v_1\in \mbox{\rm ид}_F (H\cap N)$.
\end{theorem}
\begin{proof}Ясно, что если $v\equiv v_0 + v_1\mod{M}$, $v_0\in H$, $v_1\in \mbox{\rm ид}_F (H\cap N)$, то
формула (\ref{alg1_tm2_1}) верна.
Докажем обратное утверждение теоремы.\\
Выберем в $U(F)$ стандартный базис, состоящий из слов (\ref{alg1_1}). Из (\ref{alg1_tm2_1}) следует, что $v$ --- линейная комбинация стандартных одночленов, для которых
$\mu+\nu+\eta = 1$, откуда $v\in H+N$.
Выберем $v_0\in H$ так, чтобы было $v-v_0\in N$. Тогда
\begin{eqnarray}\label{alg1_tm2_2}
\underset{m\in I\cap K} \sum D_m(v-v_0) + \underset{k\in  J\cap K} \sum g_kD_k(v-v_0)\equiv 0\mod{N_U}.
\end{eqnarray}
Ввиду леммы \ref{alg1_tm1} из (\ref{alg1_tm2_2}) следует существование $v_1\in\mbox{ид}_F (H\cap N)$
такого, что
\begin{eqnarray*}
D_k(v-v_0)\equiv D_k(v_1)\mod{N_U},~k\in K.
\end{eqnarray*}
Тогда
\begin{eqnarray*}
D_k(v-v_0 - v_1)\equiv ~0\mod{N_U},~k\in I\cup J,
\end{eqnarray*}
откуда ввиду леммы \ref{alg1_tm1_lm3} следует $v-v_0 - v_1\in M$.
\end{proof}

\noindent {\bf Следствие} \cite{Shm}.
{\sl Пусть $F$ --- свободная сумма нетривиальных алгебр Ли
$A_i~(i\in I)$ и свободной алгебры Ли с базой $\{g_j | j\in J\}$, $N$--- идеал алгебры $F$ такой, что $N\cap A_i=0~(i\in I)$. Тогда
$D_k(v)\equiv ~0\mod{N_U},~k\in I\cup J$, если и только если $v\in [N,N]$}.

\section{Теорема о свободе для свободных сумм алгебр Ли с одним определяющим соотношением}

\noindent Пусть $F$ --- свободная сумма нетривиальных алгебр Ли
$A_i$ $(i\in I)$, $T\subseteq I$, $H$ --- свободная
сумма алгебр $A_i~(i\in T)$, $N$ --- идеал алгебры $F$,
$a_1,\ldots ,a_r,\ldots$ --- базис пространства $H\cap N$.
Дополним его элементами $b_1,\ldots ,b_l,\ldots$ до базиса $N$.  Базис $N$ дополним до базиса $H+N$ элементами $e_1,\ldots ,e_k,\ldots$ из $H$. Базис $H+N$ дополним элементами
$d_1,\ldots ,d_k,\ldots $ до базиса $F$. Полагая $d_t<e_i<b_j<a_s$
получаем стандартный базис алгебры $U(F)$, состоящий из слов
\begin{eqnarray}\label{alg2_1}
d_{t_1}\ldots
d_{t_\theta}e_{i_1}\ldots
e_{i_\eta}b_{j_1}\ldots b_{j_\nu}a_{s_1}\ldots a_{s_\mu},
\end{eqnarray}
где $t_1\leqslant\ldots\leqslant t_\theta,~i_1\leqslant\ldots\leqslant i_\eta,
~j_1\leqslant\ldots\leqslant j_\nu,~s_1\leqslant\ldots\leqslant s_\mu$;
$\theta,~\eta,~\nu,~\mu\geqslant 0$.\\
Упорядочим стандартный базис алгебры $U(F)$ при помощи отношения $\leqslant$ так, что одночлены меньшей длины предшествуют
одночленам большей длины, а одночлены равной длины упорядочены лексикографически (слева направо).\\
Обозначим через $S_\beta$ --- множество стандартных одночленов, для которых $\theta> 0$ и $\nu = \mu = 0$,
через $S_\alpha$ --- множество стандартных одночленов, для которых $\theta = \nu =\mu = 0$.
Будем называть $S=S_\alpha\cup S_\beta$ системой представителей алгебры $U(F)$ по идеалу $N_U$.\\
Пусть $N=N_1 \geqslant \ldots \geqslant N_t \geqslant \ldots$ --- ряд
идеалов алгебры $F$, $[N_i,N_j\,]\leqslant N_{i+j}$, $\varphi$ --- естественный гомоморфизм $U(F)\rightarrow U(F/N_l)$, $h_1=\varphi(H\cap N)$, $h_i=h_1\cap \varphi(N_i)$.
Пусть $a_{l-1,1}\ldots ,a_{l-1,r},\ldots$ --- базис пространства $h_{l-1}$.
Дополним его элементами $b_{l-1,1},\ldots ,b_{l-1,r},\ldots$ до базиса $\varphi(N_{l-1})$.
Продолжая этот процесс применительно к идеалу $N_{l-2}$, мы построим в конце концов базис пространства $\varphi(N)$.
Если $s< t$ или $s= t$ и $\bar{s}< \bar{t}$, то полагаем $(s,\bar{s})< (t,\bar{t})$. Полагая $b_{j,\bar{j}}<a_{s,\bar{s}}$
получаем стандартный базис алгебры $U(N/N_l)$, состоящий из слов
\begin{eqnarray}\label{alg2_1_0}
b_{j_1,\overline{j_1}}\ldots b_{j_\nu,\overline{j_\nu}}a_{s_1,\overline{s_1}}\ldots a_{s_\mu,\overline{s_\mu}},
\end{eqnarray}
где $(j_1,\overline{j_1})\leqslant\ldots\leqslant (j_\nu,\overline{j_\nu})$, $(s_1,\overline{s_1})\leqslant\ldots\leqslant (s_\mu,\overline{s_\mu})$,
$\nu\geqslant 0$, $\mu\geqslant 0$.\\
Упорядочим стандартный базис алгебры $U(N/N_l)$ при помощи отношения $\leqslant$ так, что одночлены меньшей длины предшествуют
одночленам большей длины, а одночлены равной длины упорядочены лексикографически (слева направо).\\
Пусть $\delta_i$ --- идеал в $U(N/N_l)$, порожденный $\{\varphi(N_{i_1})\cdots \varphi(N_{i_s})|\,i_1 + \cdots + i_s\geqslant i\}$, $\delta_i^\prime$ --- идеал в $U(h_1)$, порожденный $\{h_{i_1}\cdots h_{i_s}|\,i_1 + \cdots + i_s\geqslant i\}$.\\
Так как элемент алгебры $U(h_1)\cap \delta_i$ равен линейной комбинации стандартных одночленов, у которых
$\nu = 0$ и $s_1 + \cdots + s_\mu\geqslant i$, то $U(h_1)\cap \delta_i= \delta_i^\prime$.\\
Если $u\in \delta_i\setminus \delta_{i+1}$, то
элемент $u$ по модулю $\delta_{i+1}$ равен линейной комбинации стандартных одночленов, у которых
$j_1 + \cdots + j_\nu + s_1 + \cdots + s_\mu= i$. Обозначим через $\alpha$ ненулевой коэффициент, с которым
входит в эту линейную комбинацию старший из стандартных одночленов.\\
Если $v\in \delta_j \setminus \delta_{j+1}$, то
элемент $v$ по модулю $\delta_{j+1}$ равен линейной комбинации стандартных одночленов, у которых
$j_1 + \cdots + j_\nu + s_1 + \cdots + s_\mu= j$. Обозначим через $\beta$ ненулевой коэффициент, с которым
входит в эту линейную комбинацию старший из стандартных одночленов.\\
Тогда $uv$ по модулю $\delta_{i+j+1}$ равен линейной комбинации стандартных одночленов, у которых
$j_1 + \cdots + j_\nu + s_1 + \cdots + s_\mu= i+j$. Так как
коэффициент, с которым входит в эту линейную комбинацию старший из стандартных одночленов, равен $\alpha\beta$, то
$uv\in \delta_{i+j}\setminus \delta_{i+j+1}$.\\
Покажем, что если $H_1=H\cap N$, $H_i=H\cap N_i$, $\Delta_i$ --- идеал, порожденный $N_{i_1}\cdots N_{i_t}$ в $U(N)$, $\Delta_i^\prime$ --- идеал, порожденный
$H_{i_1}\cdots H_{i_t}$ в $U(H_1)$, где $i_1 + \cdots + i_t\geqslant i$, то
\begin{eqnarray}\label{fm2_pr_gr}
\varphi(U(H_1))\cap \varphi(\Delta_i)=\varphi(\Delta_i^\prime).
\end{eqnarray}
Обозначим $\varphi(H_1)$ --- через $h_1$.
Полагаем $h_i=h_1\cap \varphi(N_i)$;
$\delta_i$, $\delta_i^\prime$ --- идеалы, порожденные
$\varphi(N_{i_1})\cdots \varphi(N_{i_t})$ в $U(N/N_l)$ и
$h_{i_1}\cdots h_{i_t}$ в $U(h_1)$ соответственно.\\
Ясно, что $\delta_i=\varphi(\Delta_i)$.\\
Если $i\geqslant l$, то $\varphi(N_i)=0$; если $i<l$, то $N_i \geqslant N_l$.\\
Поэтому $\varphi(H_1)\cap \varphi(N_i)=\varphi(H_1\cap N_i)$.
Следовательно, $\delta^\prime_i=\varphi(\Delta_i^\prime)$.\\
Выше было показано, что $U(h_1)\cap \delta_i=\delta_i^\prime$. Так как $U(h_1)=\varphi(U(H_1))$, то справедлива формула (\ref{fm2_pr_gr}).\\
Пусть $\alpha$, $\beta$ --- элементы алгебры $U(F)$. Ясно, что $\alpha\equiv \beta\mod{(N_l)_U}$
тогда и только тогда, когда при естественном гомоморфизме $\varphi : U(F)\rightarrow U(F/N_l)$
образы элементов $\alpha$, $\beta$ равны. Поэтому из (\ref{fm2_pr_gr}) следует, что
$U(H_1)\cap \Delta_i\equiv \Delta_i^\prime\mod{(N_l)_U}$.

\begin{lemma}\label{lm6_2_alg}
Пусть $F$ --- свободная сумма нетривиальных алгебр Ли
$A_i$ $(i\in I)$, $T\subseteq I$, $H$ --- свободная
сумма алгебр $A_i~(i\in T)$,
$N=N_1 \geqslant \ldots \geqslant N_t \geqslant \ldots $ --- ряд
идеалов алгебры $F$, $[N_i,N_j\,]\leqslant N_{i+j}$, $S=S_\alpha\cup S_\beta$ --- система представителей алгебры $U(F)$ по идеалу $N_U$.
Пусть, далее, $g_j, ~f_p,~h_k$ --- стандартные одночлены из $S_\alpha$ $(g_s\neq g_t,~f_s\neq f_t,~h_s\neq h_t\text{ при }s\neq t)$, $\Delta_k$ --- идеал в $U(N)$, порожденный $\{N_{i_1}\cdots N_{i_s}|\,i_1 + \cdots + i_s\geqslant k\}$,
\begin{eqnarray*}
v = \sum_j g_j \mu_j,~\mu_j\in U(H\cap N);
\end{eqnarray*}
\begin{eqnarray*}
r \equiv \sum_p f_p \nu_p\mod{U(F)\Delta_i},~\nu_p\in \Delta_{t-1}\setminus \Delta_t;
\end{eqnarray*}
\begin{eqnarray*}
w \equiv \sum_k h_k \lambda_k\mod{U(F)\Delta_{l-i+1}},~\lambda_k\in \Delta_{l-t}\setminus\Delta_{l-t+1}.
\end{eqnarray*}
Если $v \equiv rw\mod{U(F)\Delta_l}$, то $\nu_p\in U(H\cap N)+\Delta_t$.
\end{lemma}
\begin{proof}
Пусть $\varphi$ --- естественный гомоморфизм $U(F)\rightarrow U(F/N_l)$,
$\delta_t$ --- идеал в $U(N/N_l)$, порожденный $\{\varphi(N_{i_1})\cdots \varphi(N_{i_s})|\,i_1 + \cdots + i_s\geqslant t\}$, $h=\varphi(H\cap N)$.\\
Из $U(N)N_l\leqslant \Delta_t\cap\Delta_{l-t+1}$ следует, что $\varphi(\nu_p)\in \delta_{t-1}\setminus \delta_t$, $\varphi(\lambda_k)\in \delta_{l-t}\setminus \delta_{l-t+1}$.
Определим функцию $\psi$ на элементах $U(N/N_l)$.
Если $u$ --- слово вида (\ref{alg2_1_0}), то полагаем $\psi(u)=\nu$. Если $u=z_1m_1+\ldots+z_km_k$, $z_1,\ldots,z_k$ --- ненулевые элементы поля $P$, $m_1,\ldots,m_k$ --- попарно различные стандартные одночлены, то полагаем
\begin{eqnarray*}
\psi (u)=\max_p\, (\psi(m_p)).
\end{eqnarray*}
Если $q\in \Delta_k\setminus \Delta_{k+1}$ $(k<l)$, то $\bar{q}$ будет обозначать линейную комбинацию принадлежащих $\delta_k\setminus \delta_{k+1}$ стандартных одночленов такую, что $\varphi(q)\equiv \bar{\delta}\mod{\delta_{k+1}}$.
Предположим, найдутся $\nu_p$ такие, что $\nu_p\not\in U(H\cap N)+\Delta_t$.
Тогда из $U(N)N_l\leqslant \Delta_t$ следует, что $\varphi(\nu_p)\notin U(h)+\delta_t$.
Поэтому $\psi_G(\overline{\nu}_p)> 0$ для таких $\nu_p$.
Обозначим
\[
f_{p_0}=\max_p\, (f_p\mid\psi(\overline{\nu}_p\,)=M_\nu),
\]
\[
h_{k_0}=\max_k\, (h_k\mid\psi(\overline{\lambda}_k\,)=M_\lambda),
\]
где $M_\nu$ и $M_\lambda$ --- максимальные значения, принимаемые функцией $\psi$ на элементах $\overline{\nu}_p$ и $\overline{\lambda}_k$ соответственно.
Так как $\overline{\nu_{p_0}\lambda_{k_0}}$ --- линейная комбинация принадлежащих $\delta_{l-1}\setminus \delta_l$ стандартных одночленов,
$\psi(\overline{\nu_{p_0}\lambda_{k_0}})=M_\nu+M_\lambda$, $M_\nu > 0$, $\psi(\overline{\mu}_j\,)=0$, то получаем противоречие.
\end{proof}

\begin{lemma}\label{lm1_2}
Пусть $F$ --- свободная сумма нетривиальных алгебр Ли
$A_i$ $(i\in I)$, $T\subseteq I$, $H$ --- свободная
сумма алгебр $A_i~(i\in T)$, $N$ --- идеал алгебры $F$,
$S=S_\alpha\cup S_\beta$ --- система представителей алгебры $U(F)$ по идеалу $N_U$, $\delta_1,\ldots,\delta_l$, $\mu_1,\ldots,\mu_k$ --- элементы из $S$,
$\delta_i <\delta_j$, $\mu_i <\mu_j$ при $i <j$.\\
Тогда если $\{\delta_1,\ldots,\delta_l,\mu_1,\ldots,\mu_k\}\not\subseteq S_\alpha$, то
найдутся $\mu\in S_\beta,\,i_0,\,j_0$ такие, что $\mu$ входит в разложение $\delta_{i_0}\mu_{j_0}$ по базису $U(F)$
и не входит в разложение $\delta_i\mu_j$ по базису $U(F)$ при $(i_0,j_0)\neq (i,j)$.
\end{lemma}
\begin{proof}
Пусть стандартный базис алгебры $U(F)$ состоит из слов (\ref{alg2_1})
и $\{\delta_1,\ldots,\delta_l,\mu_1,\ldots,\mu_k\}\not\subseteq S_\alpha$.
Если $u$ --- стандартный одночлен, то полагаем $d(u)=\theta$.
Обозначим через $x$ --- максимальный элемент из $\{d(\delta_1),\ldots,d(\delta_l)\}$, через $y$ --- максимальный элемент из $\{d(\mu_1),\ldots,d(\mu_k)\}$. Тогда $x+y>0$. Пусть $\delta_{i_0}$ --- максимальный элемент из тех $\delta_i$, для которых
$d(\delta_i)=x$, $\mu_{j_0}$ --- максимальный элемент из тех $\mu_j$, для которых
$d(\mu_j)=y$. Если $\mu$ --- максимальный из стандартных одночленов,
входящих в разложение $\delta_{i_0}\mu_{j_0}$ по базису $U(F)$, то $d(\mu)=x+y$.
Т.е. $\mu\in S_\beta$ и $\mu$ не входит в разложение $\delta_i\mu_j$ по базису $U(F)$ при $(i_0,j_0)\neq (i,j)$.
\end{proof}

\begin{lemma}\label{lm2_2}
Пусть $F$ --- свободная алгебра Ли с базой $\{g_j | j\in J\}$,
$U(F)$ --- универсальная обертывающая алгебра, $U_0(F)$ --- идеал, порожденный $F$ в $U(F)$, $v\in F$. Тогда и только тогда $v\in F_{(n)}\setminus F_{(n+1)}$, когда
$v\in U_0(F)^n\setminus  U_0(F)^{n+1}$.
\end{lemma}
\begin{proof}
Известно, что $U(F)$ --- свободная ассоциативная алгебра с единицей и с базой $\{g_j | j\in J\}$, откуда следует утверждение леммы.
\end{proof}

\begin{lemma}\label{lm3_2}
Пусть $F$ --- свободная алгебра Ли, $g_1,\ldots,g_n,\ldots$ --- база $F$, $U(F)$ --- универсальная обертывающая алгебра, $U_0(F)$ --- идеал, порожденный $F$ в $U(F)$, $v\in F$, $D_1,\ldots,D_n,\ldots$ --- производные Фокса алгебры $F$.
Если $D_1(v)\notin U_0(F)^{k-1}$, то $D_1([v,g_2])\notin U_0(F)^k$.
\end{lemma}
\begin{proof}
Ясно, что $D_1([v,g_2])=D_1(v)g_2$.
Так как $U(F)$ --- свободная ассоциативная алгебра с единицей и с базой $g_1,\ldots,g_n,\ldots$, то $D_1([v,g_2])\notin U_0(F)^k$.
\end{proof}

\begin{proposition}\label{tm4}
Пусть $F$ --- свободная
сумма нетривиальных алгебр Ли $A_i$,
$N$ --- допускающий элементарные эндоморфизмы алгебры $F$ идеал алгебры $F$ такой, что $N\cap A_i=0$ $(i\in I;~|I|>2)$,
\begin{eqnarray}\label{tm4_0}
N=N_{11} \geqslant \ldots \geqslant N_{1,m_1+1}=N_{21} \geqslant \ldots \geqslant N_{y,m_y+1},
\end{eqnarray}
где $N_{kl}$ --- l-я степень алгебры $N_{k1}$.
Пусть, далее, $K\subset I$, $|K|>1$, $H$ --- свободная сумма алгебр $A_i~(i\in K)$, $R$ --- идеал алгебры $F$, $R\leqslant N$.
Если $H\cap (R+N_{1j})\neq H\cap N_{1j}$, то $H\cap (R+N_{kl})\neq H\cap N_{kl}\, (N_{kl}\leqslant N_{1j})$.
\end{proposition}
\begin{proof}
Отметим, что $H\cap N_{kl}=(H\cap N_{k1})_{(l)}$.\\
Ясно, что $H\cap N_{kl}\supseteq (H\cap N_{k1})_{(l)}$.
Пусть $\phi$ --- эндоморфизм алгебры $F$ такой, что
$\phi(a)=0$ $(a\in A_i,\,i\not\in K)$, $\phi(a)=a$  $(a\in A_i,\,i\in K)$; $u\in H\cap N_{kl}$.
Тогда $\phi(N_{kl})=(H\cap N_{k1})_{(l)}$, поэтому $u=\phi(u)\in (H\cap N_{k1})_{(l)}$.\\
Следовательно, $(H\cap N_{k1})_{(l)}\supseteq H\cap N_{kl}$.\\
Предположим, $H\cap (R+N_{1j})\neq H\cap N_{1j}$. Покажем, что
\begin{eqnarray}\label{pr1}
H\cap (R+N_{1l})> (H\cap N)_{(l)}\,~(N_{1l}\leqslant N_{1j}).
\end{eqnarray}
По условию, $H\cap (R+N_{1j})> (H\cap N)_{(j)}$.
Пусть для всех членов ряда~(\ref{tm4_0}) от $N_{1j}$ до $N_{1l}$ включительно формула~(\ref{pr1}) верна ($l\geqslant j$).
Обозначим через $B$ алгебру $H\cap N$ и через $U_0(B)$ --- идеал, порожденный $B$ в $U(B)$.
Выберем в алгебре $B$ базу $x_1,\ldots,x_m,\ldots$; $\partial_1,\ldots,\partial_m,\ldots$ --- соответствующие этой базе
производные Фокса алгебры $B$.
Пусть $v\in (H\cap (R+N_{1l}))\setminus B_{(l)}$.
Мы можем и будем считать, что
$\partial_1(v)\notin U_0(B)^{l-1}$.
Полагаем $w=[v,x_2]$.
По лемме~\ref{lm3_2}, $\partial_1(w)\notin U_0(B)^l$.\\
Следовательно, $w\in (H\cap (R+N_{1,l+1}))\setminus B_{(l+1)}$.\\
Теперь соображения индукции заканчивают доказательство формулы~(\ref{pr1}).
Из (\ref{pr1}) следует $H\cap (R+N_{21})> H\cap N_{21}$.\\
Остается заметить, что из $H\cap (R+N_{k1})> H\cap N_{k1}$ следует\\
$(H\cap (R+N_{k1}))_{(l)}> (H\cap N_{k1})_{(l)}~(l\in {\bf N})$, т.е.
$H\cap (R+N_{kl})> H\cap N_{kl}\, (N_{kl}\leqslant N_{21})$.
\end{proof}

\begin{proposition}\label{tm2}
Пусть $F$ --- свободная
сумма нетривиальных алгебр Ли $A_i$,
$N$ --- допускающий элементарные эндоморфизмы алгебры $F$ идеал алгебры $F$ такой, что $N\cap A_i=0$ $(i=1,\ldots,n)$,
\begin{eqnarray}\label{tm2_0}
N=N_{11} \geqslant \ldots \geqslant N_{1,m_1+1}=N_{21} \geqslant \ldots \geqslant N_{s,m_s+1},
\end{eqnarray}
где $N_{kl}$ --- l-я степень алгебры $N_{k1}$.\\
Пусть, далее, $r\in N_{1i}\backslash N_{1,i+1}\,(i\leqslant m_1)$, $R$ --- идеал, порожденный в алгебре $F$ элементом $r$, $H$ --- свободная
сумма алгебр $A_i~(i=1,\ldots,n-1)$.\\
Если $H\cap (R+N_{21})=H\cap N_{21}$, то $H\cap (R+N_{kl})=H\cap N_{kl}\,(k> 1)$.
\end{proposition}
\begin{proof}
\noindent Отметим, что $H\cap N_{kl}=(H\cap N_{k1})_{(l)}$.
Пусть $D_1,\ldots,D_n$ --- производные Фокса алгебры $F$. Докажем, что $D_n(r)\not\equiv 0\mod {(R+N_{21})_U}$. Предположим противное.
Тогда по теореме~\ref{alg1_tm2} нашлись бы $v_1,\ldots,v_d$ из $H\cap (R+N_{21})$; $f_1,\ldots,f_d$ из $F$
такие, что
\begin{eqnarray*}
r\equiv [v_1,f_1]+\ldots +[v_d,f_d]\mod [R+N_{21},R+N_{21}].
\end{eqnarray*}
По условию $H\cap (R+N_{21})=H\cap N_{21}$, поэтому $v_1,\ldots,v_d$ принадлежат $N_{21}$. Следовательно, $r\in N_{1,i+1}$ и мы получаем противоречие.\\
Пусть $H\cap (R+N_{kl})=H\cap N_{kl}$ для всех членов ряда (\ref{tm2_0}) от $N_{21}$ до $N_{kl}$
включительно $(N_{21} \geqslant N_{kl}$, $l\leqslant m_k)$. Требуется доказать, что $H\cap (R+N_{k,l+1})=H\cap N_{k,l+1}$.\\
Обозначим через $B$ идеал $R+N_{k1}$, через $S=S_\alpha\cup S_\beta$ --- систему представителей алгебры $U(F)$ по идеалу $B_U$,
через $U_0(B)$ --- идеал, порожденный $B$ в $U(B)$.
Выберем $v\in H\cap (R+N_{k,l+1})$, в алгебре $U(F)$ --- стандартный базис, состоящий из слов вида
(\ref{alg2_1}), в $B$ --- базу $\{x_{kz}\}$ и
обозначим через $\{\partial_{kz}\}$ соответствующие этой базе производные Фокса алгебры $B$ $(z=1,\ldots,m,\ldots)$.\\
Найдутся $u\in B_{(l+1)}$, $k_i\in U(B)$, $\mu_i\in S$, $i=1,\ldots,d$ $(\mu_i\neq \mu_j\mbox{ при }i\neq j)$ такие, что
\begin{eqnarray}\label{tm2_3}
D_m(v)\equiv D_m(r)\cdot  \sum_{i={1}}^{d} \mu_i k_i +\sum_z D_m(x_{kz})\partial_{kz}(u)\mod{(R+N_{kl})_U},
\end{eqnarray}
$m=1,\ldots,n$. Будем иметь
\begin{eqnarray}\label{tm2_4}
0\equiv D_n(r)\cdot  \sum_{i={1}}^{d} \mu_i k_i+ \sum_z D_n(x_{kz})\partial_{kz}(u)\mod{(R+N_{kl})_U}.
\end{eqnarray}
Из $D_n(r)\not\equiv 0\mod {B_U}$ следует, что
\begin{eqnarray*}
D_n(r) = \sum_{i={1}}^q  \delta_i t_i ,
\end{eqnarray*}
где $t_i\in U(B)$, $\delta_i\in S$, $i=1,\ldots,q$ $(\delta_i\neq \delta_j\mbox{ при }i\neq j)$ и
не все элементы из $\{t_1,\ldots,t_q\}$ принадлежат $U_0(B)$.
Пусть $t_{j_1},\ldots,t_{j_b}$ не принадлежат $U_0(B)$, $\delta_{j_1}<\ldots <\delta_{j_b}$.
Выберем максимальное $l_0$ такое, что $k_i\in U_0(B)^{l_0}\mod{(R+N_{kl})_U}$, $i=1,\ldots,d$.
Покажем, что $l_0\geqslant l$. Предположим противное.
Пусть  $k_{i_1},\ldots,k_{i_a}$ не принадлежат $U_0(B)^{l_0 +1}\mod{(R+N_{kl})_U}$,
$\mu_{i_1}<\ldots <\mu_{i_a}$.\\
Если $\mu$ --- максимальный из стандартных одночленов,
входящих в разложение $\delta_{j_b}\mu_{i_a}$ по базису $U(F)$, то $\mu$ больше любого  из стандартных одночленов, входящих в разложение $\delta_{j_c}\mu_{i_d}$ по базису $U(F)$ $((c,d)\neq (b,a))$ и $\mu\in S$.
Тогда
\begin{eqnarray}\label{tm2_4_1}
D_n(r)\cdot  \sum_{i={1}}^{d} \mu_i k_i = \mu t_{ba} +  \sum_{i={1}}^{\hat{q}}  \hat{\delta_i} \hat{t_i} ,
\end{eqnarray}
где $t_{ba}\not\in U_0(B)^{l_0 +1}\mod{(R+N_{kl})_U}$, $\hat{t}_i\in U(B)$, $\hat{\delta}_i\in S, \hat{\delta}_i\neq \mu$,
$i=1,\ldots,\hat{q}$.
По лемме \ref{lm2_2}, все $\partial_{kz}(u)$ лежат в $U_0(B)^l$,
поэтому (\ref{tm2_4_1}) противоречит (\ref{tm2_4}).\\
Таким образом, $k_i\in U_0(B)^l\mod{(R+N_{kl})_U}$, $i=1,\ldots,d$.
Следовательно, (\ref{tm2_3}) можно переписать в виде
\begin{eqnarray}\label{tm2_6}
D_m(v)\equiv \sum_{i={1}}^{d_m} \mu_{im} g_{im} \mod{(R+N_{kl})_U},~m=1,\ldots,n,
\end{eqnarray}
где $\mu_{im}\in S$, $g_{im}\in U_0(B)^l$.\\
Если $x$ --- элемент базы алгебры $H\cap B$, то найдутся $j_x\in \{1,\ldots,n-1\}$, $g_x\in {\bf N}$ такие, что
$D_{j_x}(x) \equiv \sum_{p=1}^{g_x} \lambda_p t_p\mod{B_U}$, где $0\neq \lambda_p\in P$, $t_p \in S_\alpha$.
Элементу $x$ поставим в соответствие строку $(M(x),j_x)$, где $M(x)$ --- произвольный элемент из $\{t_1,\ldots,t_{g_x}\}$.
Пусть $z_1,\ldots,z_p$ --- попарно различные элементы базы алгебры $H\cap B$ такие, что $v$ принадлежит алгебре, порожденной этими элементами, $(M(z_i),j_i)$ --- строка, поставленная в соответствие элементу $z_i$ описанным выше способом.
Если в $z_2,\ldots,z_p$ найдется элемент $z_i$, которому может быть поставлена в соответствие строка равная $(M(z_1),j_1)$, то
выберем $\gamma\in P$ так, чтобы элементу $z_i^\prime=z_i-\gamma z_1$ нельзя было поставить в соответствие строку равную $(M(z_1),j_1)$.
Заменим $z_i$ на $z_i^\prime=z_i-\gamma z_1$ и элементу $z_i^\prime$ поставим в соответствие строку $(M(z_i^\prime),j_i^\prime)$. Продолжая аналогичные рассуждения, найдем такую базу $X$ алгебры $H\cap B$ и такие попарно различные элементы $x_1,\ldots,x_p$ из $X$, что
$v$ принадлежит алгебре, порожденной $x_1,\ldots,x_p$, каждому $x_i$ поставлена в соответствие строка $(M(x_i),j_i)$ и строку $(M(x_i),j_i)$ нельзя поставить в соответствие $x_t$, $i+1\leqslant t\leqslant p$.
Обозначим через $\partial_z$ производные Фокса, соответствующие элементам $x_z$ $(z=1,\ldots,p)$.
Покажем, что
\begin{eqnarray}\label{tm2_6_00}
\partial_z(v)\in  U_0(B)^l\mod{(R+N_{kl})_U}.
\end{eqnarray}
Предположим, формула (\ref{tm2_6_00}) неверна.\\
Из $v\in H\cap N_{kl}$ следует $v\in (H\cap N_{k1})_{(l)}$, поэтому
$\partial_z(v)\in U_0(B)^{l-1}$.
Выберем наименьшее $i$ такое, что $\partial_i(v)\not\in U_0(B)^l\mod{(R+N_{kl})_U}$.\\
Ввиду $D_{j_i}(v)= \sum_{z=1}^p D_{j_i}(x_z)\partial_z(v)$ справедлива формула
\begin{eqnarray}\label{tm2_6_0}
D_{j_i}(v) \equiv \alpha\cdot M(x_i)\partial_i(v)+\sum_{c=1}^g t_c\lambda_c+\sum_{m={1}}^{d} \mu_m g_m \mod{(R+N_{kl})_U},
\end{eqnarray}
где $0\neq \alpha\in P$, $\lambda_c\in U(B)$, $t_c\in S$, $t_c\neq M(x_i)$,
$g_m\in U_0(B)^l$, $\mu_m\in S$.
Формула (\ref{tm2_6_0}) противоречит (\ref{tm2_6}), т.е. справедлива формула (\ref{tm2_6_00}).\\
Положим $H_t=H\cap (R+N_{kt})$. Обозначим через $U_0(H_1)$ --- идеал, порожденный $H_1$ в $U(H_1)$,
через $\Delta_l^\prime$ --- идеал, порожденный $\{H_{i_1}\cdots H_{i_s}|\,i_1 + \cdots + i_s\geqslant l\}$ в $U(H_1)$.\\
Так как $\partial_z(v)\in U(H_1)\cap U_0(B)^l\mod{(R+N_{kl})_U}$, то $\partial_z(v)\in \Delta_l^\prime\mod{(R+N_{kl})_U}$.\\
Ввиду $H\cap (R+N_{kl})=H\cap N_{kl}=(H\cap N_{k1})_{(l)}\subseteq U_0(H_1)^l$ получаем, что $\partial_z(v)\in \Delta_l^\prime$.
Если $t\leqslant l$, то $H_t=H\cap N_{kt} = (H_{k1})_{(t)}\subseteq U_0(H_1)^k$, поэтому
$\Delta_l^\prime\subseteq U_0(H_1)^l$, т.е. $\partial_z(v)\in U_0(H_1)^l$, $z=1,\ldots,p$.
Это означает, что $v\in U_0(H_1)^{l+1}$, следовательно $v\in (H\cap N_{k1})_{(l+1)}$.
Теперь соображения индукции заканчивают доказательство.
\end{proof}

\begin{proposition}\label{tm5}
Пусть $F$ --- свободная сумма нетривиальных алгебр Ли $A_i$,
$N$ --- допускающий элементарные эндоморфизмы алгебры $F$ идеал алгебры $F$ такой, что $N\cap A_i=0$ $(i=1,\ldots,n)$,
\begin{eqnarray}\label{tm5_0}
N=N_{11} \geqslant \ldots \geqslant N_{1,m_1+1}=N_{21} \geqslant \ldots \geqslant N_{s,m_s+1},
\end{eqnarray}
где $N_{kl}$ --- l-я степень алгебры $N_{k1}$.\\
Пусть, далее, $r\in N_{1i}\backslash N_{1,i+1}\,(i\leqslant m_1)$, $R$ --- идеал, порожденный в алгебре $F$ элементом $r$, $H$ --- свободная
сумма алгебр $A_i~(i=1,\ldots,n-1)$.\\
Если $r\not\in H+N_{1,i+1}$, то $H\cap (R+N_{1l})=H\cap N_{1l}\,(l=1,\ldots,m_1+1)$.
\end{proposition}
\begin{proof}
\noindent Отметим, что $H\cap N_{1l}=(H\cap N_{11})_{(l)}$.\\
Ясно, что $H\cap (R+N_{1i})=H\cap N_{1i}$.
Требуется доказать, что
\begin{eqnarray}\label{tm3_0_0}
H\cap (R+N_{1l})=H\cap N_{1l}\,(l\geqslant i).
\end{eqnarray}
Пусть для всех членов ряда (\ref{tm5_0}) от $N_{1i}$ до $N_{1l}$
включительно $(N_{1i} \geqslant N_{1l})$ формула (\ref{tm3_0_0}) верна.
Покажем, что $H\cap (R+N_{1,{l+1}})=H\cap N_{1,{l+1}}\,(l\leqslant m_1)$.\\
Обозначим через $D_1,\ldots,D_n$ --- производные Фокса алгебры $F$, $U_0(N)$ --- идеал, порожденный $N$ в $U(N)$,
через $S=S_\alpha\cup S_\beta$ --- систему представителей алгебры $U(F)$ по идеалу $N_U$.
Выберем $v\in H\cap (R+N_{1,l+1})$, в алгебре $U(F)$ --- стандартный базис, состоящий из слов вида
(\ref{alg2_1}), в $N$ --- базу $\{x_z\}$ и
обозначим через $\{\partial_z\}$ соответствующие этой базе производные Фокса алгебры $N$ $(z=1,\ldots,m,\ldots)$.\\
Если $x$ --- элемент базы алгебры $N$, то найдутся $j_x\in \{1,\ldots,n\}$, $g_x\in {\bf N}$ такие, что
$D_{j_x}(x) \equiv \sum_{p=1}^{g_x} \lambda_p t_p\mod{N_U}$, где $0\neq \lambda_p\in P$, $t_p \in S$.
Будем говорить, что элементы $t_p \in S$ входят в разложение $D_{j_x} (x)$ по модулю $N_U$.\\
Пусть $x_z$ --- элемент базы алгебры $N$.\\
а) Если $D_n (x_z)\not\equiv ~0\mod{N_U}$, то ставим в соответствие $x_z$ строку $(M(x_z),n)$, где $M(x_z)\in S$ и входит в разложение $D_n (x_z)$ по модулю $N_U$.\\
б) Если $D_n (x_z)\equiv ~0\mod{N_U}$ и найдутся $j_z\in \{1,\ldots,n-1\}$,
$M(x_z)\in S_\beta$ такие, что $M(x_z)$ входит в разложение $D_{j_z} (x_z)$ по модулю $N_U$,
то ставим в соответствие $x_z$ строку $(M(x_z),j_z)$.\\
в) Если $D_n (x_z)\equiv ~0\mod{N_U}$ и в разложения $D_k (x_z)$, $k\in \{1,\ldots,n-1\}$, по модулю $N_U$
не входят элементы из $S_\beta$, то ставим в соответствие $x_z$ строку $(M(x_z),j_z)$, где $M(x_z)\in S_\alpha$ и входит в разложение $D_{j_z} (x_z)$ по модулю $N_U$. В этом случае по лемме~\ref{alg1_lm1}
найдется $y\in H\cap N$ такой, что $D_j (x)\equiv D_j (y)\mod{N_U}$, $j\in \{1,2,\ldots,n\}$.\\
Следовательно, $x\in H\cap N\mod{[N,N]}$ \cite{Shm}.\\
Элемент $x_z$ базы алгебры $N$ будем называть $\alpha$-порождающим алгебры $N$, если ему поставлена в соответствие строка
$(M(x_z),j_z)$ описанным выше способом в) и $\beta$-порождающим алгебры $N$ в противном случае.\\
Пусть $w\in N$, $Z$ --- множество свободных порождающих алгебры $N$, $z_1,\ldots,z_p$ --- попарно различные элементы из $Z$, такие, что $w$ принадлежит алгебре, порожденной этими элементами, $(M(z_1),j_1)$ --- строка, поставленная в соответствие элементу $z_1$.
Если в $z_2,\ldots,z_p$ найдется элемент $z_k$, которому может быть поставлена в соответствие строка равная $(M(z_1),j_1)$, то
выберем $\gamma\in P$ так, чтобы $M(z_1)$ не входил в разложение
$D_{j_1} (z_k-\gamma z_1)$ по модулю $N_U$. Заменим $z_k$ на $z_k^\prime=z_k-\gamma z_1$ и элементу $z_k^\prime$ поставим в соответствие строку $(M(z_k^\prime),j_k^\prime)$. Продолжая аналогичные рассуждения, найдем такое множество $T$ свободных порождающих алгебры $N$ и такие попарно различные элементы $t_1,\ldots,t_p$ из $T$, что
$w$ принадлежит алгебре, порожденной $t_1,\ldots,t_p$, каждому $t_k$ поставлена в соответствие строка $(M(t_k),j_k)$ такая, что $M(t_k)$ входит в разложение
$D_{j_k} (t_k)$ по модулю $N_U$ и $M(t_k)$ не входит в разложение
$D_{j_k} (t_l)$ по модулю $N_U$, $k+1\leqslant l\leqslant p$.\\
Мы можем и будем считать, что $r$ и $v$ принадлежат алгебре, порожденной элементами $x_1,\ldots,x_a$,
в соответствие $x_k$ поставлена строка $(M(x_k),j_k)$ и
эту строку нельзя поставить в соответствие $x_t~(k+1\leqslant t\leqslant a)$.\\
Найдутся $u\in N_{(l+1)}$, $k_p\in U(N)$, $\mu_p\in S$, $p=1,\ldots,d$ $(\mu_p\neq \mu_t\mbox{ при }p\neq t)$ такие, что
\begin{eqnarray}\label{tm3_3}
D_m(v)\equiv D_m(r)\cdot  \sum_{p={1}}^{d} \mu_p k_p +\sum_z D_m(x_z)\partial_z(u)\mod{(R+N_{1l})_U},
\end{eqnarray}
$m=1,\ldots,n$. Будем иметь
\begin{eqnarray}\label{tm3_4}
0\equiv D_n(r)\cdot  \sum_{p={1}}^{d} \mu_p k_p+ \sum_z D_n(x_z)\partial_z(u)\mod{(R+N_{1l})_U}.
\end{eqnarray}
Идеал в $U(N)$, порожденный $\{(R+N_{1i_1})\cdots (R+N_{1i_s})|\,i_1 + \cdots + i_s\geqslant k\}$,
обозначим через $\Delta_k$.
Отметим, что если $k\leqslant i$, то $U_0(N)^k=\Delta_k$.\\
По лемме \ref{lm2_2}, $\partial_z(u)\in U_0(N)^l$, $\partial_z(v)\in U_0(N)^{l-1}$, $\partial_z(r)\in U_0(N)^{i-1}$.\\
Рассмотрим случай $\{\partial_z(v) | z=1,\ldots,a\} \not\subseteq \Delta_l\mod{(R+N_{1l})_U}$.
Выберем минимальное $k$ такое, что $\partial_k(r) \in \Delta_{i-1}\setminus \Delta_i$.\\
Ввиду $D_{j_k}(r)= \sum_{z=1}^a D_{j_k}(x_z)\partial_z(r)$ справедлива формула
\begin{eqnarray*}
D_{j_k}(r) \equiv \alpha\cdot M(x_k)\partial_k(r)+\sum_{p=1}^g t_p \lambda_p \mod{U(F)\Delta_i},
\end{eqnarray*}
где $0\neq \alpha\in P$, $\lambda_p\in \Delta_{i-1}$, $t_p\in S$, $t_p\neq M(x_k)$.\\
Т.е. $D_{j_k}(r) =\delta_1 a_1+ \cdots +\delta_q a_q$, $a_p \in \Delta_{i-1}$ и $a_p \in \Delta_{i-1}\setminus \Delta_i$ для некоторых $p$; $\delta_t\in S$, $\delta_t\neq \delta_p$ при $t\neq p$. Затем, выбрав минимальное $k^\prime$ такое, что $\partial_{k^\prime}(v) \in \Delta_{l-1}\setminus \Delta_l$,
аналогичными рассуждениями докажем, что $D_{j_{k^\prime}}(v) =\delta^\prime_1 a^\prime_1+ \cdots +\delta^\prime_{q^\prime} a^\prime_{q^\prime}$, $a^\prime_p \in \Delta_{l-1}$ и $a^\prime_p \in \Delta_{l-1}\setminus \Delta_l$ для некоторых $p$, $\delta^\prime_t\in S$, $\delta^\prime_t\neq \delta^\prime_p$ при $t\neq p$.\\
Тогда из (\ref{tm3_3}) следует, что $k_p \in \Delta_{l-i}$, $p\in \{1,\ldots,d\}$ и
$k_p \in \Delta_{l-i}\setminus \Delta_{l-i+1}$ для некоторых $p$.
Пусть $A$ --- подмножество в $\{1,\ldots,d\}$ такое, что $k_p \in \Delta_{l-i}\setminus \Delta_{l-i+1}$ при $p\in A$,
$C$ --- подмножество в $\{1,\ldots,q\}$ такое, что $a_p \in \Delta_{i-1}\setminus \Delta_i$ при $p\in C$.
Так как $D_m(v)$ --- сумма элементов вида $\gamma b$, где $\gamma\in S_\alpha$, $b\in U(N)\cap U(H)$,
то (\ref{tm3_3}) показывает, что не существует таких $\mu\in S_\beta$, $k_0 \in A$, $p_0 \in C$, что $\mu$ входит в разложение $\delta_{k_0}\mu_{p_0}$ по базису $U(F)$ и не входит в разложение $\delta_k\mu_p$ по базису $U(F)$,
$k \in A$, $p \in C$, $(k_0,p_0)\neq (k,p)$.
Тогда, ввиду леммы \ref{lm1_2}, если $k\in A$, то $\mu_k\in S_\alpha$, т.е.
\begin{eqnarray}\label{tm3_4_1_1}
\sum_{p={1}}^{d} \mu_p k_p\equiv \hat{\mu}_1 b_1+ \cdots +\hat{\mu}_{\hat{d}} b_{\hat{d}}\mod{U(F)\Delta_{l-i+1}},
\end{eqnarray}
$\hat{\mu}_k\in S_\alpha$, $\hat{\mu}_k\neq \hat{\mu}_p$ при $k\neq p$, $b_k\in \Delta_{l-i}\setminus \Delta_{l-i+1}$.\\
Покажем, что если $x_z$ --- $\beta$-порождающий, то
\begin{eqnarray}\label{tm3_4_1_gralg}
\partial_z(r)\in U_0(N)^i,~z\in\{1,\ldots,a\}.
\end{eqnarray}
Предположим противное. Пусть $b$ --- минимальное число, для которого формула (\ref{tm3_4_1_gralg}) неверна,
т.е. $\partial_b(r)\in U_0(N)^{i-1}\setminus U_0(N)^i$.\\
Рассмотрим случай, когда $x_b$ --- $\beta$-порождающий и ему поставлена в соответствие строка $(M(x_b),n)$.
В этом случае:\\
если $x_c$ --- $\beta$-порождающий и $c<b$, то $\partial_c(r)\in U_0(N)^i$,\\
если $x_c$ --- $\alpha$-порождающий, то $D_n(x_c)\equiv 0\mod{N_U}$, откуда $D_n(x_c)\partial_c(r)\in U(F)\Delta_i$.\\
Тогда из $D_n(r)=\sum_z D_n(x_z)\partial_z(r)$ следует, что
\begin{eqnarray}\label{tm3_4_2_alggr}
D_n(r)\equiv \alpha\cdot M(x_b)\partial_b(r)+\sum_{p=1}^g t_p\lambda_p\mod{U(F)\Delta_i},
\end{eqnarray}
где  $0\neq \alpha\in P$, $\lambda_p\in \Delta_{i-1}$, $t_p\in S$, $t_p\neq M(x_b)$. Из
(\ref{tm3_4_1_1}), (\ref{tm3_4_2_alggr}) следует, что
\begin{eqnarray}\label{tm3_4_2_gralg}
D_n(r)\cdot    \sum_{p={1}}^{d} k_p\mu_p\equiv\hat{\delta}_1 \hat{a}_1+ \cdots +\hat{\delta}_s \hat{a}_s\mod{(R+N_{1l})_U},
\end{eqnarray}
$\hat{\delta}_k\in S$, $\hat{\delta}_k\neq \hat{\delta}_p$ при $k\neq p$, $\{\hat{a}_1,\ldots,\hat{a}_s\}\subset U(N)$ и $\hat{a}_p \in \Delta_{l-1}\setminus \Delta_l$ для некоторых $p$.
Формула (\ref{tm3_4_2_gralg}) противоречит (\ref{tm3_4}).\\
Рассмотрим случай, когда $x_b$ --- $\beta$-порождающий и ему поставлена в соответствие строка $(M(x_b),t)$, $t\neq n$, $M(x_b)\in S_\beta$. Тогда
\begin{eqnarray*}
D_t(r)\equiv  \alpha\cdot M(x_b)\partial_b(r)+\sum_{p=1}^g t_p\lambda_p\mod{U(F)\Delta_i},
\end{eqnarray*}
где  $0\neq \alpha\in P$, $\lambda_p\in \Delta_{i-1}$, $t_p\in S$, $t_p\neq M(x_b)$. Будем иметь
\begin{eqnarray}\label{tm3_4_3_gralg}
D_t(r)\cdot  \sum_{p={1}}^{d} k_p\mu_p\equiv\hat{\delta}_1 \hat{a}_1+ \cdots +\hat{\delta}_s \hat{a}_s\mod{(R+N_{1l})_U},
\end{eqnarray}
$\hat{\delta}_k\in S$, $\hat{\delta}_k\neq \hat{\delta}_p$ при $k\neq p$, $\{\hat{a}_1,\ldots,\hat{a}_s\}\subset U(N)$ и, для некоторых $p$, $\hat{\delta}_p\in S_\beta$, $\hat{a}_p \in \Delta_{l-1}\setminus \Delta_l$.
Формула (\ref{tm3_4_3_gralg}) противоречит (\ref{tm3_3}).\\
Полученные противоречия показывают, что формула (\ref{tm3_4_1_gralg}) верна.\\
Покажем, что если $x_z$ --- $\alpha$-порождающий, то $\partial_z(r)\in U(N\cap H)\mod{U_0(N)^i}$.\\
Предположим противное. Пусть $b$ --- минимальное число, для которого $x_b$ --- $\alpha$-порождающий и
$\partial_z(r)\not\in U(N\cap H)\mod{U_0(N)^i}$. В этом случае:\\
если $x_c$ --- $\alpha$-порождающий и $c<b$, то $\partial_c(r)\in U(N\cap H)\mod{U_0(N)^i}$,\\
если $x_c$ --- $\beta$-порождающий, то ввиду (\ref{tm3_4_1_gralg}) $\partial_c(r)\in U_0(N)^i$.\\
Пусть элементу $x_b$ поставлена в соответствие строка $(M(x_b),t)$. Тогда
\begin{eqnarray}\label{tm3_4_1_2_gralg}
D_t(r)\equiv M(x_b) a_0 + \delta_1 a_1+ \cdots +\delta_q a_q\mod{U(F)\Delta_i},
\end{eqnarray}
где $a_0\not\in U(N\cap H) \mod{U_0(N)^i}$;
$M(x_b),\,\delta_k\in S_\alpha$; $\delta_k\neq M(x_b)$; $\delta_k\neq \delta_p$ при $k\neq p$; $a_k\in \Delta_{i-1}\setminus \Delta_i$.\\
Так как $D_{j_t}(v)$ --- сумма элементов вида $\gamma b$, где $\gamma\in S_\alpha$, $b\in U(H\cap N)$, то из
\begin{eqnarray*}
D_t(v)\equiv (M(x_b) a_0 + \delta_1 a_1+ \cdots +\delta_q a_q)\cdot  (\hat{\mu}_1 b_1 + \cdots + \hat{\mu}_{\hat{d}} b_{\hat{d}})\mod{U(F)\Delta_l},
\end{eqnarray*}
следует ввиду (\ref{tm3_4_1_1}), (\ref{tm3_4_1_2_gralg}) и леммы \ref{lm6_2_alg}, что
$a_0\in U(N\cap H) \mod{U_0(N)^i}$. Противоречие.\\
Следовательно, для всех $\alpha$-порождающих $x_z$
\begin{eqnarray}\label{tm3_4_1_0_gralg}
\partial_z(r)\in U(N\cap H)\mod{U_0(N)^i}.
\end{eqnarray}
Т.е. если $k\in \{1,\ldots,a\}$, то либо $\partial_k(r)\in U_0(N)^i$, либо $D_m(x_k)$ --- линейная комбинация элементов из $S_{\alpha}\mod{N_U}$, $m=1,\ldots,n$, $\partial_k(r)\in U(H\cap N)\mod{U_0(N)^i}$ и $D_n(x_k)\in N_U$.
Пусть $\psi\text{: }U(F)\rightarrow U(F)$ --- эндоморфизм, оставляющий на месте элементы из $H$, отображающий в $0$
элементы из $A_n$. Обозначим $\psi(r)$ через $\hat{r}$.
Ясно, что $\hat{r}\in H\cap N$.\\
Тогда $\hat{r}= \sum_{t={1}}^a \psi(x_t)\psi(\partial_t(r))$, где либо $\partial_t(r)$, $\psi(\partial_t(r))\in \Delta_i$, либо
\begin{eqnarray}\label{main_009}
\psi(x_t)\equiv x_t\mod{[N,N]};\,\psi(\partial_t(r))\equiv  \partial_t(r)\mod{\Delta_i}.
\end{eqnarray}
Если $\partial_t(r)$, $\psi(\partial_t(r))\in \Delta_i$, то $x_t \partial_t(r)$, $\psi(x_t) \psi(\partial_t(r))\in U_0(N)^{i+1}$.\\
Если для $x_t$ верны соотношения (\ref{main_009}), то
$\psi(x_t) \psi(\partial_t(r))\equiv x_t \partial_t(r)\mod{U_0(N)^{i+1}}$
(ввиду $[N,N]\subseteq U_0(N)^2$ и $\partial_t(r)\in U_0(N)^{i-1}$).
Следовательно, $r -\hat{r}\in U_0(N)^{i+1}$, т.е. $r\in H+N_{1,i+1}$. Противоречие.
Полученное противоречие показывает, что $\{\partial_z(v) | z=1,\ldots,a\} \subseteq \Delta_l\mod{(R+N_{1l})_U}$.\\
Положим $H_k=H\cap (R+N_{1k})$. Обозначим через $U_0(H_1)$ идеал, порожденный $H_1$ в $U(H_1)$,
через $\Delta_l^\prime$ --- идеал, порожденный $\{H_{i_1}\cdots H_{i_p}|i_1 + \cdots + i_p\geqslant l\}$ в $U(H_1)$.\\
Так как $U(H_1)\cap \Delta_l\equiv \Delta_l^\prime\mod{(R+N_{1l})_U}$ и
$\partial_z(v)\in U(H_1)\cap\Delta_l\mod{(R+N_{1l})_U}$, то
$\partial_z(v)\in \Delta_l^\prime\mod{(R+N_{1l})_U}$.\\
Ввиду $H\cap (R+N_{1l})=H\cap N_{1l}=(H\cap N_{11})_{(l)}\subseteq U_0(H_1)^l$ получаем, что $\partial_z(v)\in \Delta_l^\prime$.
Если $k\leqslant l$, то $H_k=H\cap N_{1k} = (H_1)_{(k)}\subseteq U_0(H_1)^k$, поэтому
$\Delta_l^\prime\subseteq U_0(H_1)^l$, т.е. $\partial_z(v)\in U_0(H_1)^l$, $z=1,\ldots,a$.
Это означает, что $v\in U_0(H_1)^{l+1}$, следовательно $v\in (H\cap N_{11})_{(l+1)}$.
Теперь соображения индукции заканчивают доказательство.
\end{proof}

\noindent Из предложений \ref{tm4}, \ref{tm2}, \ref{tm5}, вытекает справедливость теоремы~\ref{tm2_algr}.

\section{Теорема о свободе для свободных сумм алгебр Ли с конечным числом определяющих соотношений}

\noindent Пусть $G$ --- алгебра Ли. Следующие преобразования матрицы над $U(G)$ назовем элементарными:
\begin{align}
&\text{Перестановка столбцов $i$ и $j$;}\label{gr2_df1}\\
&\text{Перестановка строк $i$ и $j$;}\label{gr2_df2}\\
&\text{Умножение справа }i\text{-й строки на ненулевой элемент из }U(G);\label{gr2_df3}\\
&\text{Прибавление к }j\text{-й строке }i\text{-й строки,}\label{gr2_df4}\\
&\text{умноженной справа на ненулевой элемент из }U(G)\text{, где }i<j.\notag
\end{align}
Если $M=\|a_{kn}\|$ --- матрица над $U(G)$, $\Phi$ --- последовательность
элементарных преобразований матрицы $M$, то $\Phi(M)$ будет обозначать матрицу, полученную из $M$
последовательностью преобразований $\Phi$.
Матрицу $M$ будем называть треугольной ранга $t$, если
$a_{kk}\neq 0\,(k\leqslant t)$, $a_{kn}=0\,(n<k\leqslant t)$, $a_{kn}=0\,(k> t)$.

\begin{lemma}\label{lm4_2_gr_1}
Пусть $F$ --- алгебра Ли,
$N=N_1 \geqslant \ldots \geqslant N_m\geqslant \ldots$ --- ряд идеалов алгебры $F$, $[N_p\,,N_q\,]\leqslant N_{p+q}$,
$F/N$ --- разрешимая алгебра.\\
Пусть, далее, $\phi$ --- естественный гомоморфизм $U(F)\to U(F/N_m)$, $\Delta_t$ --- идеал, порожденный
$\{\phi(N_{i_1})\cdots \phi(N_{i_c})|i_1 + \cdots + i_c\geqslant t\}$ в $U(F/N_m)$,
$\|a_{kn}\|$ --- $r\times s$ матрица над $U(F/N_m)$, $\phi^\prime$ --- естественный гомоморфизм $U(F/N_m)\to U(F/N)$, $\|\phi^\prime(a_{kn})\|$ --- треугольная $r\times s$ матрица над $U(F/N)$, $\psi$ --- функция на $U(F/N_m)$, такая, что $\psi(0)=\infty$; если $\phi^\prime(\alpha)\neq 0$, то $\psi(\alpha)=0$;
если $\alpha\in \Delta_j\setminus \Delta_{j+1}$, то $\psi(\alpha)=j$.
Тогда матрицу $\|a_{kn}\|$ последовательностью элементарных преобразований {\rm (\ref{gr2_df3})}, {\rm (\ref{gr2_df4})}
можно привести к треугольной $r\times s$ матрице $\|b_{kn}\|$ такой, что $\psi(b_{kk})\leqslant \psi(b_{kn})$.
\end{lemma}
\begin{proof}
\noindent Отметим, что в $U(F/N)$, $U(F/N_m)$ нет делителей нуля.
Так как $F/N_m$ --- разрешимая алгебра, то в $U(F/N_m)$ выполняется правое условие Оре (доказательство см., например, в \cite{Hm}).
Обозначим через $\delta_t$ идеал, порожденный
$\{\phi(N_{i_1})\cdots \phi(N_{i_c})|i_1 + \cdots + i_c\geqslant t\}$ в $U(N/N_m)$,
через $S$ --- систему представителей алгебры $U(F/N_m)$ по подалгебре $U(N/N_m)$ такую, что $S\cap (N/N_m)_U =\emptyset$.\\
Тогда $\Delta_t=S\delta_t$ и если $\alpha\in S\delta_i\setminus S\delta_{i+1}$, $\beta\in S\delta_j\setminus S\delta_{j+1}$,
то $\alpha\beta\in S\delta_{i+j}\setminus S\delta_{i+j+1}$.\\
Отсюда вытекает, что $\psi(\alpha\beta)=\psi(\alpha)+\psi(\beta)$. Ясно, что $\psi(\alpha+\beta)\geqslant\min\{\psi(\alpha),\psi(\beta)\}$, т.е. функция $\psi$ является нормированием на $U(F/N_m)$.\\
По условию, $\phi^\prime(a_{kk})\neq 0$, а если $n<k$, то $\phi^\prime(a_{kn})= 0$,
поэтому $\psi(a_{kk})=0$, а если $n<k$, то $\psi(a_{kn})>0$, $k=1,\ldots,r$.
Обозначим $i$-ю строку матрицы $\|a_{kn}\|$ через $v_i=(a_{i1},\ldots,a_{is})$.
Нужную нам матрицу $\|b_{kn}\|$ будем строить индукцией.
Полагаем $b_{1j}=a_{1j}$, $j=1,\ldots,s$, $\bar{v}_1=(b_{11},\ldots,b_{1s})$.
Ясно, что $\psi(b_{11})\leqslant \psi(b_{1j})$, $j=1,\ldots,s$.
Пусть последовательностью элементарных преобразований {\rm (\ref{gr2_df3})}, {\rm (\ref{gr2_df4})}
строк $v_1,\ldots,v_{t-1}$ матрицы $\|a_{kn}\|$
построены строки $\bar{v}_1,\ldots,\bar{v}_{t-1}$ матрицы $\|b_{kn}\|$
такие, что при $n<k$ $b_{kn}= 0$ и $\psi(b_{kk})\leqslant \psi(b_{kn})$, $n=1,\ldots,s$,
$a_{t1}=\ldots=a_{tl}=0$.\\
Если $l=t-1$,
то полагаем $b_{tj}=a_{tj}$, $j=1,\ldots,s$, $\bar{v}_t=(b_{t1},\ldots,b_{ts})$.
Ясно, что $\psi(b_{tt})\leqslant \psi(b_{tj})$, $j=1,\ldots,s$.
Рассмотрим случай $l<t-1$, $a_{t,l+1}\neq 0$, $\phi^\prime(a_{t,l+1})= 0$. Выберем ненулевые $\beta_1,\,\beta_2\in U(F/N_m)$ так, чтобы было $b_{l+1,l+1}\beta_1=-a_{t,l+1}\beta_2$.
Тогда $\bar{v}_{l+1}\beta_1+v_t\beta_2=(c_{t1},\ldots,c_{tr})$, $c_{t1}=\ldots=c_{t,l+1}= 0$.
Из $\psi(b_{l+1,j}\beta_1)\geqslant \psi(b_{l+1,l+1}\beta_1)= \psi(a_{t,l+1}\beta_2) >\psi(\beta_2)$,
$\psi(a_{tt}\beta_2)=\psi(\beta_2)$ следует, что $\psi(c_{tt})=\psi(\beta_2)$,
$\psi(c_{tt})\leqslant \psi(c_{tj})$, $j=1,\ldots,s$. Отметим, что если $j<t$, то $\psi(c_{tt})< \psi(c_{tj})$.
Продолжая аналогичные рассуждения, построим из строки $v_t$ последовательностью элементарных преобразований {\rm (\ref{gr2_df3})}, {\rm (\ref{gr2_df4})} строку $\bar{v}_t=(b_{t1},\ldots,b_{ts})$ такую, что $b_{t1}=\ldots=b_{t,t-1}=0$ и $\psi(b_{tt})\leqslant \psi(b_{tj})$, $j=1,\ldots,s$. Теперь соображения индукции заканчивают доказательство.
\end{proof}
\noindent В дополнение к лемме \ref{lm4_2_gr_1} отметим, что преобразованиями {\rm (\ref{gr2_df1})}, {\rm (\ref{gr2_df2})},
{\rm (\ref{gr2_df3})}, {\rm (\ref{gr2_df4})} произвольную матрицу $\|a_{kn}\|$ над $U(F/N_m)$ можно привести к треугольной матрице $\|b_{kn}\|$ такой, что $\psi(b_{kk})\leqslant \psi(b_{kn})$.\\
Действительно, элементарными преобразованиями
можно добиться, чтобы на месте $(1,1)$ был элемент $a_{ij}$ такой, что $\psi(a_{ij})=M$,
где $M$ --- минимальное значение, принимаемое функцией $\psi$ на элементах $a_{kn}$,
а затем с его помощью получить нули в первом столбце нашей матрицы.
Продолжая этот процесс применительно к строкам и столбцам у которых номера больше единицы,
мы дойдем в конце концов до треугольной матрицы $\|b_{kn}\|$ такой, что $\psi(b_{kk})\leqslant \psi(b_{kn})$.
\begin{lemma}\label{lm4_2_gr_2}
Пусть $G$ --- разрешимая алгебра Ли; $M$  --- матрица над
$U(G)$, $\alpha_i$ --- $i$-я строка матрицы $M$ ($i=1,\ldots,t$);
$\psi$ --- элементарное преобразование матрицы $M$; $M^\psi$  --- матрица, полученная из $M$ преобразованием
$\psi$; $\alpha_1^\psi,\ldots,\alpha_t^\psi$ --- строки матрицы $M^\psi$;
$\alpha$ --- правая линейная комбинация строк $\alpha_1,\ldots,\alpha_t$.
Если $\psi$ --- преобразование строк матрицы $M$, то полагаем $\alpha^\psi=\alpha$;
если $\psi$ --- перестановка столбцов $i$ и $j$ матрицы $M$, то полагаем $\alpha^\psi$ --- строка,
полученная из $\alpha$ перестановкой $i$-й и $j$-й координат.
Тогда найдется ненулевой элемент $d\in U(G)$ такой, что $\alpha^\psi d$  --- правая линейная комбинация строк $\alpha_1^\psi,\ldots,\alpha_t^\psi$.
\end{lemma}
\begin{proof}
Ясно, что можем ограничиться рассмотрением случая, когда
$M^\psi$ получена из $M$ умножением справа i-й строки матрицы $M$ на ненулевой элемент $a\in U(G)$.
По условию, найдутся элементы $b_1,\ldots,b_t$ из $U(G)$ такие, что
$\alpha_1b_1+\ldots+\alpha_tb_t=\alpha$. Если $b_i= 0$, то берем $d=1$.
Предположим $b_i\neq 0$.  Так как $G$ --- разрешимая алгебра Ли, то в $U(G)$ выполняется правое условие Оре (доказательство см., например, в \cite{Hm}). Поэтому найдутся ненулевые $c$, $d\in U(G)$ такие, что $ac=b_id$
и мы получим $\alpha_1b_1d+\ldots+\alpha_i ac+\ldots+\alpha_tb_td=\alpha\, d$.
\end{proof}

\noindent Доказательство теоремы \ref{tm3_algr}.
Если $n-m=1$, $H=A_i$ ($i$ --- произвольно выбранный элемент из \{1,\ldots,n\}), то $H\cap (R+N_{kl}) = H\cap N_{kl}=0$, где
$N_{kl}$ --- произвольный член ряда {\rm (\ref{end_algr_3})}.
Поэтому будем предполагать, что $n-m>1$.\\
Естественные гомоморфизмы $U(F)\to U(F/(R+N_{k,m_k+1}))$ и $U(F/(R+N_{k,m_k+1}))\to U(F/(R+N_{k1}))$
обозначим через $\phi_k$ и $\phi^\prime_k$ соответственно.
Идеалы, порожденные $\{\phi_k(R+N_{ki_1})\cdots \phi_k(R+N_{ki_j})|i_1 + \cdots + i_j\geqslant t\}$ в $U(F/(R+N_{k,m_k+1}))$
и $U((R+N_{k1})/(R+N_{k,m_k+1}))$ обозначим через $\Delta_{kt}$ и $\Delta_{kt}^\prime$
соответственно. Пусть $\psi_k$ --- функция на $U(F/(R+N_{k,m_k+1}))$ такая, что
$\psi_k(0)=\infty$; если $\phi^\prime_k(\alpha)\neq 0$, то $\psi_k(\alpha)=0$;
если $\alpha\in \Delta_{kj}\setminus \Delta_{k,j+1}$, то $\psi_k(\alpha)=j$.
Естественный гомоморфизм $U(F)\to U(F/N)$ обозначим через $\phi_0$.
Пусть $\psi_0$ --- функция на $U(F/N)$ такая, что
$\psi_0(0)=\infty$; если $\alpha\neq 0$, то $\psi_0(\alpha)=0$.\\
При доказательстве леммы \ref{lm4_2_gr_1} было показано, что функция $\psi$ является нормированием на $U(F/N_m)$. Аналогично доказывается, что функция $\psi_k$ является нормированием на $U(F/(R+N_{k,m_k+1}))$,
т.е. если $\alpha,\beta\in U(F/(R+N_{k,m_k+1}))$, то $\psi_k(\alpha\beta)=\psi_k(\alpha)+\psi_k(\beta)$, $\psi_k(\alpha+\beta)\geqslant\min\{\psi_k(\alpha),\psi_k(\beta)\}$.
Ясно, что функция $\psi_0$ является нормированием на $U(F/N)$.\\
Пусть $D_1,\ldots,D_n$ --- производные Фокса алгебры $U(F)$.\\
Обозначим $D_j(r_i)$ через $m_{ij}$, через $M$ --- матрицу $\|m_{ij}\|$, через $M^{\phi_k}$ --- матрицу $\|\phi_k(m_{ij})\|$, через $t_k$ --- ранг $M^{\phi_k}$.\\
Если $r_1,\ldots,r_m$ --- элементы из $N_{s,m_s+1}$, то $R+N_{kl} = N_{kl}$, где $N_{kl}$ --- произвольный член ряда {\rm (\ref{end_algr_3})}, поэтому будем считать, что в $\{0,\ldots,s\}$ найдется $K$ такое, что $t_K>0$ и $t_i=0$ при $i<K$.
Элементарными преобразованиями  приведем матрицу $M^{\phi_K}$ к треугольной матрице $\|m^{(K)}_{ij}\|$ такой, что $\psi_K(m^{(K)}_{ii})\leqslant \psi_K(m^{(K)}_{ij})$.\\
Обозначим через $M_K$ матрицу $\|m^{(K)}_{ij}\|$, через $\Phi_K$
последовательность элементарных преобразований матрицы $M$ такую, что $(\Phi_K(M))^{\phi_K}=M_K$.\\
Пусть $K<k\leqslant s$; $\Phi_{k-1}$ --- последовательность элементарных преобразований матрицы $M$; $(\Phi_{k-1}(M))^{\phi_{k-1}}=M_{k-1}=\|m^{(k-1)}_{ij}\|$ --- треугольная матрица
такая, что $\psi_{k-1}(m^{(k-1)}_{ii})\leqslant \psi_{k-1}(m^{(k-1)}_{ij})$.
Полагаем $M_{k1}=\Phi_{k-1}(M)$.\\
Так как $(M_{k1}^{\phi_k})^{\phi^\prime_k}=M_{k-1}$, то ввиду леммы~\ref{lm4_2_gr_1} найдется такая последовательность $\Phi_{k1}$ элементарных преобразований (\ref{gr2_df3}), (\ref{gr2_df4})
первых $t_{k-1}$ строк матрицы $M_{k1}$, что элементы матрицы $(\Phi_{k1}(M_{k1}))^{\phi_k}=\|b_{ij}\|$ будут
удовлетворять условиям $b_{ij}=0$ при $j< i$ и $\psi_k(b_{ii})\leqslant \psi_k(b_{ij})$, $i=1,\ldots,t_{k-1}$.
Полагаем $M_{k2}=\Phi_{k1}(M_{k1})$.
Последовательностью $\Phi_{k2}$ элементарных преобразований (\ref{gr2_df3}), (\ref{gr2_df4}) с помощью $1\text{-й},\ldots,t_{k-1}\text{-й}$ строк матрицы $M_{k2}$ добьемся того, что в $1\text{-м},\ldots,t_{k-1}\text{-м}$ столбцах матрицы $(\Phi_{k2}(M_{k2}))^{\phi_k}$ под диагональю будут нули.
Полагаем $M_{k3}=\Phi_{k2}(M_{k2})$.
Последовательностью $\Phi_{k3}$ элементарных преобразований строк и столбцов матрицы $M_{k3}$, номера которых больше $t_{k-1}$, добьемся того, что $(\Phi_{k3}(M_{k3}))^{\phi_k}=\|m^{(k)}_{ij}\|$ будет треугольной  матрицей и $\psi_k(m^{(k)}_{ii})\leqslant \psi_k(m^{(k)}_{ij})$.
Обозначим через $M_k$ матрицу $\|m^{(k)}_{ij}\|$, через $\Phi_k$
последовательность$\Phi_{k-1}$, $\Phi_{k1}$, $\Phi_{k2}$, $\Phi_{k3}$.
Из индуктивных соображений можно считать, что построены $M_k$, $\Phi_k$, $k=K,\ldots,s$.
Полагаем $I_s$ --- множество $i_1,\ldots,i_{t_s}$ номеров столбцов $m_{i_1},\ldots,m_{i_{t_s}}$ матрицы $M$ таких, что $\Phi_s(m_{i_j})$ --- $j$-й столбец матрицы $\Phi_s(M)$; $\{j_1,\ldots,j_p\}= \{1,\ldots,n\}\setminus I_s$;
$H$ --- свободная сумма алгебр $A_i,~i\in \{j_1,\ldots,j_p\}$.
Так как $t_s\leqslant m$, то $p\geqslant n-m$. Пусть $\Phi$ --- последовательность элементарных преобразований матрицы $M$, $\varphi\in\Phi$, $v_1,\ldots,v_n\in U(F)$.
Будем считать, что если $\varphi$ --- преобразование строк матрицы $M$, то $\varphi$ действуют на строку $(v_1,\ldots,v_n)$ тождественно, а если $\varphi$ --- перестановка столбцов $i$ и $j$ матрицы $M$, то $\varphi$ действуют на строку $(v_1,\ldots,v_n)$
перестановкой элементов $v_i$ и $v_j$.\\
Обозначим через $B$ алгебру $R+N_{k1}$, через $U_0(B)$ --- идеал, порожденный $B$ в $U(B)$,
через $\{x_{kz}\}$ --- базу алгебры $B$,
через $\{\partial_{kz}\}$ --- соответствующие этой базе производные Фокса алгебры $U(B)$ $(z=1,\ldots,l,\ldots)$.
Выберем $v\in H\cap (R+N_{k,m_k+1})$.
Найдутся $u\in N_{k,m_k+1}$, $\beta_1,\ldots,\beta_m\in U(F)$,
такие, что
\begin{eqnarray}\label{tm2_3_gr_end}
D_j(v)\equiv \sum_{i={1}}^{m} D_j(r_i)\beta_i+D_j(u)\mod{(R+N_{k,m_k+1})_U}\,,~j= 1,\ldots,n.
\end{eqnarray}
Так как $u\in (N_{k1})_{(m_k+1)}$, то $\partial_{kz}(u)\in U_0(B)^{m_k}$.
Из $D_j(u)= \sum_z D_j(x_{kz})\partial_{kz}(u)$ следует $D_j(u)\in U(F)U_0(B)^{m_k},~j=1,\ldots,n$. Покажем, что
\begin{eqnarray}\label{tm2_6_gr_end}
\phi_k(D_j(v))\in \Delta_{km_k},~j=1,\ldots,n.
\end{eqnarray}
Если $t_k=0$, то $D_j(r_i)\equiv 0\mod{(R+N_{k,m_k+1})_U}$, поэтому (\ref{tm2_6_gr_end}) вытекает из (\ref{tm2_3_gr_end}).
Пусть $t_k>0$, $V=(D_1(v)-D_1(u),\ldots,D_n(v)-D_n(u))$, $\overline{V}=(-D_1(u),\ldots,-D_n(u))$.
Отметим, что элементы строки $\overline{V}\,^{\phi_k}$ лежат в $\Delta_{km_k}$.
Ввиду леммы~\ref{lm4_2_gr_2} и (\ref{tm2_3_gr_end}) найдется элемент $d\in U(F)$ такой, что $\phi_k(d)\neq 0$ и строка $(\Phi_k(V d))^{\phi_k}$
линейно выражается через строки $1,\ldots,t_k$ матрицы $M_k$.\\
Обозначим через $\gamma_1,\ldots,\gamma_{t_k}$ элементы $U(F)$ такие, что $\phi_k(\gamma_i)$ --- коэффициент при
$i$-й строке матрицы $M_k$ в записи $(\Phi_k(V d))^{\phi_k}$ через строки  $1,\ldots,t_k$ матрицы $M_k$.
Первые $t_k$ элементов строки $(\Phi_k(V d))^{\phi_k}$ лежат в $\Delta_{km_k} \phi_k(d)$, так как они соответственно равны первым $t_k$ элементам строки $(\Phi_k(\overline{V} d))^{\phi_k}$.
Так как $\psi_k(m^{(k)}_{ii})\leqslant \psi_k(m^{(k)}_{ij})$, то все элементы строки $(\Phi_k(V d))^{\phi_k}$  лежат в $\Delta_{km_k}\phi_k(d)$, т.е. формула (\ref{tm2_6_gr_end}) справедлива.
(Вначале замечаем, что $m^{(k)}_{11}\phi_k(\gamma_1)\in\Delta_{km_k} \phi_k(d)$,
откуда $m^{(k)}_{1j}\phi_k(\gamma_1)\in\Delta_{km_k} \phi_k(d),~j= 2,\ldots,n$. Тогда $m^{(k)}_{22}\phi_k(\gamma_2)\in\Delta_{km_k} \phi_k(d)$,
откуда $m^{(k)}_{2j}\phi_k(\gamma_2)\in\Delta_{km_k} \phi_k(d),~j= 3,\ldots,n$ и т.д.)\\
Ясно, что $H\cap (R+N_{11}) = H\cap N_{11}$. Предположим, для всех $N_{ij}$ от $N_{11}$ до $N_{kl}$ включительно $H\cap (R+N_{ij}) = H\cap N_{ij}$, $l\leqslant m_k$.
Возьмем $v\in H\cap (R + N_{k,l+1})$.\\
Рассмотрим случай $l=1$.
Будем иметь $H\cap (R+N_{k1})=H\cap N_{k1}$ и $v\in H\cap (R + N_{k2})$.
Найдутся $u\in N_{k2}$, $\beta_1,\ldots,\beta_m\in U(F)$, такие, что
\begin{eqnarray*}
D_j(v)\equiv \sum_{i={1}}^{m} D_j(r_i)\beta_i+D_j(u)\mod{(R+N_{k1})_U}\,,~j= 1,\ldots,n.
\end{eqnarray*}
Пусть $V=(D_1(v)-D_1(u),\ldots,D_n(v)-D_n(u))$.
Ввиду леммы~\ref{lm4_2_gr_2} найдется элемент $d\in U(F)$ такой, что $\phi_{k-1}(d)\neq 0$ и строка
$(\Phi_{k-1}(V d))^{\phi_{k-1}}$
линейно выражается через строки треугольной матрицы $M_{k-1}$ ранга $t_{k-1}$.
Так как $u\in N_{k2}$, то $\phi_{k-1}((D_j(u))=0$ $(j=1,\ldots,n)$,
поэтому в строке $(\Phi_{k-1}(V d))^{\phi_{k-1}}$ первые $t_{k-1}$ элементов --- нули, следовательно,
строка $(\Phi_{k-1}(V d))^{\phi_{k-1}}$ --- нулевая.\\
Тогда $D_j(v)\equiv 0 \mod{(N_{k1})_U}$ $(j=1,\ldots,n)$,
откуда $v\in N_{k2}$ \cite{Shm}.\\
Рассмотрим случай $l>1$. Обозначим через $S=S_\alpha\cup S_\beta$ систему представителей алгебры $U(F)$ по идеалу $B_U$.
Если $z$ --- элемент базы алгебры $H\cap B=H\cap N_{k1}$, то найдутся $x\in \{1,\ldots,p\}$, $g_x\in {\bf N}$ такие, что
$D_{j_x}(z) \equiv \sum_{i=1}^{g_x} \lambda_i t_i\mod{B_U}$, где $0\neq \lambda_i\in P$, $t_i \in S_\alpha$.
Элементу $z$ поставим в соответствие строку $(M(z),j_x)$, где $M(z)$ --- произвольный элемент из $\{t_1,\ldots,t_{g_x}\}$.
Пусть $z_1,\ldots,z_q$ --- попарно различные элементы базы алгебры $H\cap N_{k1}$ такие, что $v$ принадлежит алгебре, порожденной этими элементами, $(M(z_i),q_i)$ --- строка, поставленная в соответствие элементу $z_i$ описанным выше способом $(q_i\in \{j_1,\ldots,j_p\})$.
Если в $z_2,\ldots,z_p$ найдется элемент $z_i$, которому может быть поставлена в соответствие строка равная $(M(z_1),q_1)$, то
выберем $\gamma\in P$ так, чтобы элементу $z_i^\prime=z_i-\gamma z_1$ нельзя было поставить в соответствие строку равную $(M(z_1),q_1)$.
Заменим $z_i$ на $z_i^\prime=z_i-\gamma z_1$ и элементу $z_i^\prime$ поставим в соответствие строку $(M(z_i^\prime),q_i^\prime)$.
Продолжая аналогичные рассуждения, найдем такую базу $X$ алгебры $H\cap N_{k1}$ и такие попарно различные элементы $x_1,\ldots,x_q$ из $X$, что $v$ принадлежит алгебре, порожденной $x_1,\ldots,x_q$, каждому $x_i$ поставлена в соответствие строка $(M(x_i),q_i)$ и строку $(M(x_i),q_i)$ нельзя поставить в соответствие $x_t$, $t\neq i$.\\
Обозначим через $\partial_1,\ldots,\partial_q$ производные Фокса, соответствующие $x_1,\ldots,x_q$.\\
Положим $H_t=H\cap (R+N_{kt})$. Обозначим через $U_0(H_1)$ --- идеал, порожденный $H_1$ в $U(H_1)$,
через $\Delta_t$ --- идеал, порожденный $\{(R+N_{ki_1})\cdots (R+N_{ki_s})|i_1 + \cdots + i_s\geqslant t\}$ в $U(R+N_{k1})$
через $\Delta_t^\prime$ --- идеал, порожденный $\{H_{i_1}\cdots H_{i_s}|i_1 + \cdots + i_s\geqslant t\}$ в $U(H_1)$.\\
Если $h\in U(H_1)$ и $\phi_k(h)\in \Delta_{k2}^\prime$, то $h\in U(H_1)\cap \Delta_2\mod{(R+N_{k,m_k+1})_U}$.\\
Поэтому $h\in \Delta_2^\prime \mod{(R+N_{k2})_U}$ (ввиду $(R+N_{k,m_k+1})_U\subseteq (R+N_{k2})_U$).
Из $H_2=H\cap (R+N_{k2})=H\cap N_{k2}=(H\cap N_{k1})_{(2)}\subseteq U_0(H_1)^2$ следует, что $\Delta_2^\prime \subseteq U_0(H_1)^2$
и $h\in U_0(H_1)^2$.
Если  $x$ --- элемент базы алгебры $H_1$, то $x\notin U_0(H_1)^2$, следовательно $\phi_k(x)\notin \Delta_{k2}^\prime$.
Покажем, что
\begin{eqnarray}\label{tm2_6_gr_end_0}
\phi_k(\partial_z(v))\in \Delta_{kl}^\prime,~z=1,\ldots,q.
\end{eqnarray}
Предположим, формула (\ref{tm2_6_gr_end_0}) неверна.\\
Выберем $i$ такое, что $\phi_k(\partial_i(v))\in \Delta_{k,l-1}^\prime\setminus \Delta_{kl}^\prime$.
Обозначим через $\bar{v}$ элемент $[v,x_t]$ $(t\neq i)$.
Так как $\bar{v}\in (H\cap N_{k1})_{(l+1)}$ (ввиду $v\in H\cap N_{kl}=(H\cap N_{k1})_{(l)}$), а $x_t$ --- элемент базы алгебры $H\cap N_{k1}$, то $\partial_z(\bar{v})\in U_0(H_1)^{l}$, $z=1,\ldots,q$,
$\phi_k(x_t)\in \Delta_{k1}^\prime\setminus \Delta_{k2}^\prime$.\\
Из $\phi_k(\partial_i(\bar{v}))=\phi_k(\partial_i(v)x_t)$ следует, что $\phi_k(\partial_i(\bar{v}))\in \Delta_{kl}^\prime\setminus \Delta_{k,l+1}^\prime$.\\
Ввиду $D_{q_i}(\bar{v})= \sum_{z=1}^q D_{q_i}(x_z)\partial_z(\bar{v})$ справедлива формула
\begin{eqnarray*}
D_{q_i}(\bar{v}) = \alpha\cdot M(x_i)\partial_i(\bar{v})+\sum_{c=1}^g t_c\lambda_c+\sum_{d={1}}^{h} \mu_d g_d,
\end{eqnarray*}
где $0\neq \alpha\in P$; $t_c,\mu_d\in S$, $t_c\neq M(x_i)$; $\lambda_c\in U(B)$,
$g_d\in U_0(B)^{l+1}$.\\
Поэтому $\phi_k(D_{q_i}(\bar{v}))\not\in \Delta_{k,l+1}$ $(\bar{v}\in H\cap (R+N_{k,l+2})$.\\
Аналогичными рассуждениями найдем такой элемент $w\in H\cap (R+N_{k,m_k+1})$, что
$\phi_k(D_{q_i}(w))\not\in \Delta_{km_k}$ --- противоречие с (\ref{tm2_6_gr_end}). Следовательно, формула (\ref{tm2_6_gr_end_0}) верна.\\
Так как $U(H_1)\cap \Delta_l\equiv \Delta_l^\prime\mod{(R+N_{kl})_U}$ и $(R+N_{k,m_k+1})_U\subseteq(R+N_{kl})_U$, то ввиду (\ref{tm2_6_gr_end_0})
$\partial_z(v)\in \Delta_l^\prime\mod{(R+N_{kl})_U}$.\\
Ввиду $H\cap (R+N_{kl})=H\cap N_{kl}=(H\cap N_{k1})_{(l)}\subseteq U_0(H_1)^l$ получаем, что $\partial_z(v)\in \Delta_l^\prime$.
Если $t\leqslant l$, то $H_t=H\cap N_{kt} = (H_1)_{(t)}\subseteq U_0(H_1)^t$, поэтому
$\Delta_l^\prime\subseteq U_0(H_1)^l$, т.е. $\partial_z(v)\in U_0(H_1)^l$, $z=1,\ldots,q$.
Это означает, что $v\in U_0(H_1)^{l+1}$, следовательно $v\in (H\cap N_{k1})_{(l+1)}$.
Теперь соображения индукции заканчивают доказательство.

\end{document}